\documentclass[11pt]{article}

\usepackage[normalem]{ulem}

\usepackage{amsmath,amssymb,latexsym,soul,amsthm}
\usepackage{amsfonts}
\usepackage{color,enumitem,graphicx}
\usepackage[colorlinks=true,urlcolor=blue,
=red,linkcolor=blue,linktocpage,pdfpagelabels,
bookmarksnumbered,bookmarksopen]{hyperref}
\usepackage[english]{babel}

\usepackage{authblk}
\usepackage[a4paper,left=1.0in,right=1.0in,top=3cm,bottom=3cm]{geometry}
\usepackage[T1]{fontenc}
\usepackage{mathpazo}
\usepackage{eulervm}
\usepackage{microtype}
\usepackage{fixmath}
\numberwithin{equation}{section}

\newtheorem{theorem}{Theorem}[section]
\newtheorem{lemma}[theorem]{Lemma}

\theoremstyle{definition}

\newtheorem{remark}[theorem]{Remark}

\newcommand{\R}{\mathbb{R}}

\usepackage{todonotes}

\newenvironment{altproof}[1]
{\noindent
{\em Proof of {#1}}.}
{\nopagebreak\mbox{}\hfill $\Box$\par\addvspace{0.5cm}}

\newcommand{\cE}{{\mathcal E}}

\newcommand{\cI}{{\mathcal I}}

\newcommand{\cN}{{\mathcal N}}

\newcommand{\cS}{{\mathcal S}}

\newcommand{\Ga}{\Gamma}

\newcommand{\weakto}{\rightharpoonup}

\numberwithin{equation}{section}

\DeclareMathOperator*{\essinf}{ess\,inf}

\DeclareMathOperator*{\supp}{supp}

\DeclareSymbolFont{rsfs}{U}{rsfs}{m}{n}
\DeclareSymbolFontAlphabet{\mathscr}{rsfs}

\begin{document}

\title{The semirelativistic Choquard equation with a local nonlinear term}

\author{Bartosz Bieganowski\thanks{Email address: \texttt{bartoszb@mat.umk.pl}}}
\affil{Nicolaus Copernicus University, Faculty of Mathematics and Computer Science, ul. Chopina 12/18, 87-100 Toru\'n, Poland}

\author{Simone Secchi\thanks{Email address: \texttt{Simone.Secchi@unimib.it}}}
\affil{Dipartimento di Matematica e Applicazioni, Universit\`a degli Studi di Milano-Bicocca, via Roberto Cozzi 55, I-20125, Milano, Italy}

\maketitle

\begin{abstract}
  We propose an existence result for the semirelativistic Choquard equation with a local nonlinearity in $\mathbb{R}^N$
  \begin{equation*}
  \sqrt{\strut -\Delta + m^2} u - mu + V(x)u =  \left( \int_{\R^N} \frac{|u(y)|^p}{|x-y|^{N-\alpha}} \, dy \right) |u|^{p-2}u - \Ga (x) |u|^{q-2}u,
\end{equation*}
  where $m > 0$ and the potential $V$ is decomposed as the sum of a $\mathbb{Z}^N$-periodic term and of a bounded term that decays at infinity.
  The result is proved by variational methods applied to an auxiliary problem in the half-space $\mathbb{R}_{+}^{N+1}$.
   \medskip


   \end{abstract}

\tableofcontents

\section{Introduction}
\setcounter{section}{1}

The semirelativistic Hartree equation
\begin{equation*}
  \mathrm{i} \frac{\partial \psi}{\partial t} = \sqrt{\strut I-\Delta}\ \psi -V_\gamma(\psi)\psi \quad \hbox{in $\mathbb{R}\times \mathbb{R}^N$}
\end{equation*}
%\todo[inline]
%{
%There should be a minus sign?
%$$
%  \mathrm{i} \frac{\partial \psi}{\partial t} = \sqrt{\strut I-\Delta}\ \psi - V_\gamma(\psi)\psi \quad \hbox{in $\mathbb{R}\times \mathbb{R}^N$}
%$$
%Then our equation is, in fact, more general. Otherwise the sign of the convolution term is different.
%}
%\todo[inline,color=yellow]{The literature is somehow obscure. I suspect that there is a phenomenon like the case of ``focusing'' or ``defocusing'' nonlinear Schr\"{o}dinger equation. Both signs are meaningful, though they describe different situations. Anyway, since we are not writing a treatise on relativistic quantum mechanics, I have changed the sign in front of $V_\gamma$.}
in the unknown $\psi=\psi(t,x)$, where
\begin{equation*}
  V_\gamma (\psi) (x)=  \left( |\cdot|^{-\gamma} * |\psi|^2 \right) (x)=  \int_{\mathbb{R}^N}
  \frac{|\psi(t,y)|^2}{|x -y|^\gamma}\, dy
\end{equation*}
appears as a model describing Boson stars, see \cite{Elgart,Frohlich,LiebYau}. For the simplicity of presentation,
the mass, speed of light and Planck constant have been normalized. Solitary wave solutions $\psi(t,x)=\mathrm{e}^{-\mathrm{i}t\lambda} u(x)$
lead to the non-local stationary equation
\begin{equation*}
\sqrt{\strut I-\Delta}\ u - V_\gamma(u)u = \lambda u \quad \hbox{in $\mathbb{R}^N$}.
\end{equation*}
In this paper we will consider the more general equation
\begin{equation}\label{eq1}
\sqrt{\strut -\Delta + m^2} u - mu + V(x)u =  \left( \int_{\R^N} \frac{|u(y)|^p}{|x-y|^{N-\alpha}} \, dy \right) |u|^{p-2}u - \Ga (x) |u|^{q-2}u,
\end{equation}
where $m > 0$ and $V = V_l + V_p \in L^\infty( \R^N)$ is decomposed so that $V_l$ is a bounded potential that vanishes at infinity,
while $V_p$ is a $\mathbb{Z}^N$-periodic function. We assume that \((N-1)p-N<\alpha<N\) and
$2 < q < \min \left\{2p, 2N/(N-1)\right\}$. We can rewrite \eqref{eq1} as
\[
\sqrt{\strut -\Delta + m^2} u + (V(x)-m)u =  \left( \int_{\R^N} \frac{|u(y)|^p}{|x-y|^{N-\alpha}} \, dy \right) |u|^{p-2}u - \Ga (x) |u|^{q-2}u.
\]
%
%\todo[inline]{
%Another interesting problem:
%\[
%\sqrt{-\Delta + m^2} u + (V(x)-m)u =  \left( \int_{\R^N} \frac{F(u)}{|x-y|^{N-\alpha}} \, dy \right) f(u) - \Ga (x) |u|^{q-2}u,
%\]
%where $f$ is subcritical and satisfies some condition, which allows us to use Theorem \ref{ThSetting}.
%}
%
%\todo[inline]{
%My proposal:\\
%$f \colon \R \rightarrow \R$ is continuous and $|f(u)| \leq c (1 + |u|^{p-1})$; \\
%$f(u)/u \to 0$ as $u \to 0$;\\
%$F(u)/|u|^q \to \infty$ as $|u| \to \infty$, where $F(u) := \int_0^u f(s) \, ds$;\\
%$u \mapsto f(u) / |u|^{q-1}$ is strictly increasing on $(-\infty,0) \cup (0,\infty)$.
%}
%
In the rest of the paper, we will write for simplicity
\[
I_\alpha = \frac{1}{|x|^{N-\alpha}}.
\]
In the case $\Gamma =0$, equation~\eqref{eq1} is also known as the Choquard-Pekard
or Schr\"{o}dinger-Newton equation and recently many papers have been devoted to the
study of solitary states and their semiclassical limit: see
\cite{a,CSM,CT,cho,ccs,ccs1,css,Cingolani,Frohlich,ls,mpt,mz,pe2,pe3,Secchi-ground,t,ww} and references therein.

\bigskip

In the rest of the paper, we will retain the following assumptions:
\begin{enumerate}
\item[(N)] \((N-1)p-N<\alpha<N\), $2 < q < \min \left\{2p, 2N/(N-1)\right\}$ and $p \geq 2$\color{black}.
\item[(V1)] The potential $V$ can be split as $V = V_p + V_l$, where
$V_p \in L^\infty (\mathbb{R}^N)$ is $\mathbb{Z}^N$-periodic and
$V_l \in L^\infty (\R^N)$ is such that either
$V_l \equiv 0$ or $V_l(x) < 0$ for a.e.~$x \in \R^N$, and
\[
\lim_{|x| \to \infty} V_l (x) = 0.
\]
Furthermore,
\begin{equation} \label{eq:1.5V}
V_l \in L^s (\mathbb{R}^N),
\end{equation}
where $N \leq s < + \infty$.
\item[(V2)] $\essinf_{x \in \R^N} V(x) > 0$.
%\color{red}
% \item[(V3)] The potential $V$ can be split as $V=V_p+V_l$, where $V_p \in L^\infty (\R^N)$ is $\mathbb{Z}^N$-periodic and $\essinf_{x\in\R^N} V_p (x) > 0$, and $V_l \in L^\infty (\R^N)$ is such that $V_l(x) > 0$ for a.e. $x \in \R^N$, and
% \[
% \lim_{|x| \to \infty} V_l (x) = 0.
% \]
% Furthermore,
% \begin{equation*}
% V_l \in L^s (\mathbb{R}^N),
% \end{equation*}
% where $N \leq s < + \infty$.
%\color{black}
\item[($\Ga$)] $\Ga \in L^\infty (\R^N)$ is $\mathbb{Z}^N$-periodic and non-negative.
\end{enumerate}

The main result of the paper is the following.
\begin{theorem} \label{th:main}
Suppose that (N), (V1), (V2) and ($\Gamma$) hold true. Then equation \eqref{eq1} has a
ground state solution $u \in H^{1/2}(\mathbb{R}^N) \cap C^\infty(\mathbb{R}^N)$.
\end{theorem}
\begin{remark}
Although we expect our solution to decay exponentially fast at infinity, we are currently unable to prove it rigorously. The lack of information about the sign of the ground state prevents us from applying standard comparison techniques like in \cite{Coti}.
\end{remark}
The proof of this result follows from a recent variational technique introduced in
\cite{BieganowskiMederski} and extended to the fractional setting in \cite{Bieganowski}.
The presence of the non-local convolution term introduces some additional difficulty.
To sketch the ideas, we first show that the Euler functional $\mathcal{E}$ associated to \eqref{eq1}
satisfies certain geometrical conditions on the Nehari manifold.
An abstract result (see \ref{ThSetting} below) yields then the existence
of a bounded Palais-Smale sequence $\{v_n\}_n$ for the $\mathcal{E}$.
A decomposition result for Palais-Smale sequences
(see Lemma \ref{lem:splitting}) implies that $\{v_n\}_n$ converges to a non-trivial solution
of \eqref{eq1}.

\medskip

Assumption (V2) is almost optimal for existence, as the following result shows.

%\color{red}
\begin{theorem}\label{th:main2}
Suppose that (N) and ($\Gamma$) hold true, while (V2) is replaced by
\begin{enumerate}
  \item[(V3)] The potential $V$ can be split as $V=V_p+V_l$, where $V_p \in L^\infty (\R^N)$ is $\mathbb{Z}^N$-periodic and $\essinf_{x\in\R^N} V_p (x) > 0$, and $V_l \in L^\infty (\R^N)$ is such that $V_l(x) > 0$ for a.e. $x \in \R^N$, and
  \[
  \lim_{|x| \to \infty} V_l (x) = 0.
  \]
  Furthermore,
  \begin{equation*}
  V_l \in L^s (\mathbb{R}^N),
  \end{equation*}
  where $N \leq s < + \infty$.
\end{enumerate}
 Then equation \eqref{eq1} has no ground state solutions.
\end{theorem}
%\color{black}

Our last result deals with compactness of ground states in the case $V_l \equiv 0$.
\begin{theorem}\label{th:main3}
Suppose that (N), (V1) and (V2) hold true and $V_l = 0$. Let $\{ \Gamma_n \}_n \subset L^\infty (\R^N)$ be a sequence such that $\Gamma_n$ satisfies ($\Gamma$) and $\Gamma_n \to 0$ in $L^\infty (\R^N)$. Let $u_n \in H^{1/2} (\R^N)$ be a ground state solution with $\Gamma = \Gamma_n$. Then there are $z_n \in \mathbb{Z}^N$ such that
\begin{equation*}
u_n (\cdot - z_n) \to u \quad \mbox{in} \ H^{1/2} (\R^N),
\end{equation*}
where $u \in H^{1/2} (\R^N)$ is a ground state solution with $\Gamma = 0$.
\end{theorem}

%\color{red}
%\begin{theorem}\label{th:main4}
%Suppose that ... There is $C > 0$ such that $|u(x)| \leq C \exp (-m |x|)$.
%\end{theorem}
%\color{black}

The paper is organized as follows. Section \ref{sect:2} describes the local realization of the nonlocal operator $\sqrt{-\Delta + m^2}$ and contains some preliminary facts. In Section \ref{sect:3} we present the variational setting. Section \ref{sect:4} provides Brezis-Lieb-type splitting results. Finally, Sections \ref{sect:5}, \ref{sect:6}, \ref{sect:7} contain proofs of Theorems \ref{th:main}, \ref{th:main2}, \ref{th:main3} respectively.

\section{Functional setting and local realization}\label{sect:2}

The realization of the operator $\sqrt{-\Delta + m^2}$ in Fourier variables is not convenient
for our purposes. Therefore, we prefer to make use of a \emph{local} realization
(see \cite{Caffarelli, Coti}) by means of the \emph{Dirichlet-to-Neumann} operator.

Given $u \in \cS (\R^N)$, the Schwartz space of rapidly decaying smooth functions defined on $\R^N$, there exists one and only one function $v \in \cS (\R_+^{N+1})$ such that
\begin{gather*}
\left\{ \begin{array}{ll}
-\Delta v  + m^2 v = 0 \ & \hbox{in} \ \R_+^{N+1}, \\
v(0,y) = u(y) \ & \hbox{for} \ y \in \R^N = \partial \R_+^{N+1}.
\end{array} \right.
\end{gather*}
Setting
\begin{gather*}
Tu(y) = - \frac{\partial v}{\partial x} (0,y),
\end{gather*}
we easily see that the problem
\begin{gather*}
\left\{ \begin{array}{ll}
-\Delta w  + m^2 w = 0 \ & \hbox{in} \ \R_+^{N+1}, \\
w(0,y) = Tu(y) \ & \hbox{for} \ y \in \R^N = \partial \R_+^{N+1}
\end{array} \right.
\end{gather*}
is solved by $w(x,y) = - \frac{\partial v}{\partial x} (x,y)$. From this we deduce that
\begin{gather*}
T(Tu)(y) = - \frac{\partial w}{\partial x} (0,y) = \frac{\partial^2 v}{\partial x^2} (0,y) = \left( -\Delta_y v + m^2 v \right) (0,y),
\end{gather*}
and hence $T \circ T = (-\Delta_y + m^2)$, namely $T$ is a square root of the Schr\"odinger operator $-\Delta_y + m^2$ on $\R^N = \partial \R_+^{N+1}$.

From the previous construction, we can replace the nonlocal problem (\ref{eq1}) in $\R^N$ with the local Neumann problem in the half-space $\R_+^{N+1}$
\begin{gather*}
\left\lbrace
\begin{array}{ll}
-\Delta v (x,y) + m^2 v(x,y) = 0 \quad \mathrm{in} \ \R_+^{N+1}, \\
-\frac{\partial v}{\partial x} (0,y) = - \left(V(y)-m\right) v(0,y) +  (I_\alpha * |v(0,\cdot)|^p ) |v(0,y)|^{p-2} v(0,y) \\
\quad \quad \quad \quad \quad \quad \quad \quad - \Gamma(y) |v(0,y)|^{q-2}v(0,y) \quad \mathrm{for} \ y \in \R^N.
\end{array}
\right.
\end{gather*}

We introduce the Sobolev space $H = H^1 ( \R_+^{N+1})$, and recall that there is a continuous \emph{trace operator} $\gamma \colon H \rightarrow H^{1 / 2} (\R^N)$. Moreover, this operator is surjective and the inequality
\begin{equation} \label{eq:1.2}
|\gamma(v)|_p^p \leq p |v|^{p-1}_{2(p-1)} \left| \frac{\partial v}{\partial x} \right|_2
\end{equation}
holds for every $v \in H$: we refer to \cite{Tartar} for basic facts about the Sobolev space $H^{1/2} (\R^N)$ and the properties of the trace operator. It follows immediately from \eqref{eq:1.2} that
\begin{equation}
\label{eq:1.3}
\int_{\mathbb{R}^N} |\gamma(v)|^2 \, dy \leq m \int_{\mathbb{R}_{+}^{N+1}} |\nabla v|^2 \, dx\, dy + \frac{1}{m} \int_{\mathbb{R}_{+}^{N+1}} |v|^2 \, dx\, dy
\end{equation}

Reasoning as in \cite[Page 5]{Cingolani} and taking the Hardy-Littlewood-Sobolev inequality (see \cite[Theorem 4.3]{Lieb}) into consideration, it follows easily that the functional $\cE \colon H \rightarrow \R$ defined by
\begin{align*}
\cE (v) &= \frac{1}{2} \int_{\R_+^{N+1}} |\nabla v|^2 \, dx \, dy + \frac{m^2}{2} \int_{\R_+^{N+1}} v^2 \, dx \, dy + \frac{1}{2} \int_{\R^N} \left( V(y) - m \right) \gamma(v)^2 \, dy \\
&\quad -  \frac{1}{2p} \int_{\R^N}  (I_\alpha * |\gamma(v)|^p) |\gamma(v)|^p \, dy + \frac{1}{q} \int_{\R^N} \Gamma(y) |\gamma(v)|^q \, dy
\end{align*}
is of class $C^1$, and its critical points are (weak) solutions to problem (\ref{eq1}). In particular,
\begin{align}\label{eq2}
\mathscr{D}(v) := \int_{\R^N}  (I_\alpha * |\gamma(v)|^p) |\gamma(v)|^p \, dy
= \int_{\R^N \times \R^N}  \frac{|\gamma(v)(y)|^p |\gamma(v)(y')|^p}{|y-y'|^{N-\alpha}} \, dy \, dy' \leq C  \|v\|^{2p}.
\end{align}
%In what follows, we make the following assumptions
%\begin{enumerate}
%\item[(V1)] The potential $V$ can be split as $V = V_p + V_l$, where
%$V_p \in L^\infty (\mathbb{R}^N)$ is $\mathbb{Z}^N$-periodic and
%$V_l \in L^\infty (\R^N)$ is such that either
%$V_l \equiv 0$ or $V_l(x) < 0$ for a.e.~$x \in \R^N$, and
%\[
%\lim_{|x| \to \infty} V_l (x) = 0.
%\]
%Furthermore,
%\begin{equation} \label{eq:1.5V}
%V_l \in L^s (\mathbb{R}^N),
%\end{equation}
%where $N \leq s < + \infty$.
%\item[(V2)] $\essinf_{x \in \R^N} V(x) > 0$.
%\item[($\Ga$)] $\Ga \in L^\infty (\R^N)$ is $\mathbb{Z}^N$-periodic and non-negative.
%\end{enumerate}

\begin{lemma}
The quadratic form
\begin{equation} \label{eq:1.5}
  v \mapsto Q(v) := \int_{\R_+^{N+1}}
  |\nabla v|^2 \, dx \, dy + m^2 \int_{\R_+^{N+1}} v^2 \, dx \, dy +
  \int_{\R^N} \left( V(y) - m \right) \gamma(v)^2 \, dy \in \R.
\end{equation}
defines a norm on $H$  that is equivalent to the standard one on $H^1(\mathbb{R}_{+}^{N+1})$.
\end{lemma}
\begin{proof}
For a real number \(z\), we will write \(z^{+}=\max\{z,0\}\geq 0\) and \(z^{-}=-\min\{z,0\} \geq 0\).
For any $v \in H$, we recall \eqref{eq:1.3} and compute
\begin{align*}
Q(v) &\geq \int_{\mathbb{R}_{+}^{N+1}} |\nabla v|^2 \, dx\, dy + m^2 \int_{\mathbb{R}_{+}^{N+1}} |v|^2 \, dx\, dy - \int_{\mathbb{R}^N} \left( V(y)-m \right)^{-} |\gamma (v)|^2 \, dy \\
&\geq \int_{\mathbb{R}_{+}^{N+1}} |\nabla v|^2 \, dx\, dy + m^2 \int_{\mathbb{R}_{+}^{N+1}} |v|^2 \, dx\, dy -  \left( \operatorname{ess\, inf}_{\mathbb{R}^N} V -m \right)^{-} \int_{\mathbb{R}^N} |\gamma (v)|^2 \, dy \\
&\geq \int_{\mathbb{R}_{+}^{N+1}} |\nabla v|^2 \, dx\, dy + m^2 \int_{\mathbb{R}_{+}^{N+1}} |v|^2 \, dx\, dy \\
&\quad {} + \left( \operatorname{ess\, inf}_{\mathbb{R}^N} V -m \right)^{-} \left(
\frac{1}{m} \int_{\mathbb{R}_{+}^{N+1}} |\nabla v|^2 \, dx\, dy + m \int_{\mathbb{R}_{+}^{N+1}} |v|^2 \, dx\, dy
\right) \\
&\geq \min \left\{ 1, \frac{\operatorname{ess\, inf}_{\mathbb{R}^N} V}{m}
\right\} \int_{\mathbb{R}_{+}^{N+1}} |\nabla v|^2 \, dx\, dy + \min \left\{ m^2 , \operatorname{ess\, inf}_{\mathbb{R}^N} V
\right\}  \int_{\mathbb{R}_{+}^{N+1}} |v|^2 \, dx\, dy.
\end{align*}
Recalling assumption (V2) we conclude that there exists a constant $C>0$ such that
\begin{equation} \label{eq:1.6}
Q(v) \geq C^{-1} \left\| v \right\|_{H^1(\mathbb{R}_{+}^{N+1})}.
\end{equation}
On the other hand, for any $v \in H$,
\begin{align*}
Q(v) &\leq \int_{\mathbb{R}_{+}^{N+1}} |\nabla v |^2 \, dx\, dy + m^2 \int_{\mathbb{R}_{+}^{N+1}} |v|^2 \, dx\, dy + \int_{\mathbb{R}^N} \left( V(y)-m \right)^{+} |\gamma (v)|^2 \, dy \\
&\leq \int_{\mathbb{R}_{+}^{N+1}} |\nabla v |^2 \, dx\, dy + m^2 \int_{\mathbb{R}_{+}^{N+1}} |v|^2 \, dx\, dy + \left( \|V\|_\infty -m \right)^{+}\int_{\mathbb{R}^N}  |\gamma (v)|^2 \, dy,
\end{align*}
and \eqref{eq:1.3} immediately yields that
\begin{equation} \label{eq:1.7}
Q(v) \leq C \left\| v \right\|_{H^1(\mathbb{R}_{+}^{N+1})}.
\end{equation}
The conclusion follows from \eqref{eq:1.6} and \eqref{eq:1.7}
\end{proof}
In the rest of the paper, we will endow $H$ with the norm \eqref{eq:1.5}.

We will repeatedly use the following integration Lemma.
\begin{lemma}\label{lemmaIntegr}
	If
	\begin{equation} \label{1.8}
	\max \left\{ 1, \frac{N}{N(2-p)+p}\right\} < t < \frac{N}{N-\alpha}
	%\frac{N}{N-\alpha} > t > \max \left\{ 1, \frac{N}{N(2-p)+p}\right\},
	\end{equation}
then~\(I_\alpha \in L^t_{\mathrm{loc}} (\R^N)\). In particular, we can decompose $I_\alpha = I_1 + I_2$, where
\begin{equation} \label{eq:decompI}
I_1 := I_\alpha \chi_{B(0,1)} \in L^t (\R^N), \quad I_2 := I_\alpha - I_1 \in L^\infty (\R^N).
\end{equation}
\end{lemma}

\begin{remark}
It can be easily checked that our results continue to hold if $I_\alpha$ is replaced by a convolution kernel~$W = W_1 + W_2$ where $W_1 \in L^t (\R^N)$, $W_2 \in L^\infty (\R^N)$ and $t$ satisfies \eqref{1.8}.
\end{remark}

For any function \(u \in H^{1/2}(\mathbb{R}^N)\), we set
\[
\phi_u (x) := \left( I_\alpha * |u|^p \right) (x) =\int_{\R^N} \frac{|u(y)|^p}{|x-y|^{N-\alpha}} \, dy.
\]
The following properties are straightforward.
\begin{lemma}\label{phiProperties}
For any $u \in H^{1/2}(\R^N)$ there hold
\begin{itemize}
\item[(i)] $\phi_{tu} = t^p \phi_u$ for any $t > 0$;
\item[(ii)] $\phi_u (\cdot + z) = \phi_{u (\cdot + z)}$ for any $z \in \mathbb{Z}^N$;
\item[(iii)] $\phi_u (x) \geq 0$ for a.e. $x \in \R^N$.
\end{itemize}
\end{lemma}

The weak sequential continuity of $\mathscr{D}'$ plays a crucial r\^{o}le in our reasoning. We collect in the next Lemma this property.

\begin{lemma}\label{weakContDprime}
Assume that $(v_n) \subset H$ is a bounded sequence such that $v_n \weakto v_0$ in $H$ and $v_n \to v_0$ a.e. in~$\R_+^{N+1}$. Then for any $\psi \in H$ there holds
\begin{equation*}
\mathscr{D}'(v_n)(\psi) \to \mathscr{D}'(v_0)(\psi).
\end{equation*}
\end{lemma}
\begin{proof}
We refer e.g. to~\cite[Lemma 3.4]{a}. See also \cite[Lemma 4.2]{Cingolani} and the proof of equation~(3.3) in \cite{CL}.
\end{proof}

\section{Variational methods for equation \eqref{eq1}}\label{sect:3}

We briefly introduce the variational setting from \cite{BieganowskiMederski} based on the Nehari manifold technique, which does not require the monotonicity condition on the nonlinearity. Let $(E, \| \cdot \|_E)$ be a Hilbert space and suppose that $\cE \colon E \rightarrow \R$ given by
\[
\cE(v) := \frac{1}{2} \|v\|^2 - \cI(v)
\]
is a functional on $E$, where $\cI$ is of class $C^1$. Define
\[
\cN := \left\{ v \in E \setminus \{ 0 \} \mid \cE'(v)(v) = 0 \right\}.
\]

\begin{theorem}[{\cite[Theorem 2.1]{BieganowskiMederski}}]\label{ThSetting}
Suppose that the following conditions hold:
\begin{enumerate}
\item[(J1)] there is $r>0$ such that $a:= \inf_{\|v\|_E=r} \cE (v) > \cE(0) = 0$;
\item[(J2)] there is $q \geq 2$ such that $\cI (t_n v_n) / t_n^q \to \infty$ for any $t_n \to\infty$ and $v_n\to v \neq 0$ as $n\to\infty$;
\item[(J3)] for $t \in (0,\infty) \setminus \{1\}$ and $v \in \cN$
\[
\frac{t^2-1}{2} \cI'(v)(v) - \cI(tv)+\cI(v) < 0;
\]
\item[(J4)] $\cE$ is coercive on $\cN$.
\end{enumerate}
Then $\inf_\cN \cE > 0$ and there exists a bounded minimizing sequence for $\cE$ on $\cN$, i.e. there is a sequence $\{v_n\}_n \subset \cN$ such that $\cE(v_n) \to \inf_\cN \cE$ and $\cE'(v_n) \to 0$.
\end{theorem}

We can check that these assumptions are satisfied in our case.

\begin{lemma}
(J1)--(J4) hold with $(E, \|\cdot\|_E) := (H, \|\cdot\|)$ and
\[
\cI (v) := \frac{1}{2p} \mathscr{D}(v) - \frac{1}{q} \int_{\R^N} \Ga (y) | \gamma(v)|^q \, dy.
\]
\end{lemma}

\begin{proof}
\begin{enumerate}
\item[(J1)] In view of \eqref{eq2}
\[
\cI(v) \leq \frac{1}{2p} \mathscr{D}(v) \leq C \| v\|^{2p} = C \|v\|^{2p-2} \|v\|^2.
\]
For $v \in H$ such that $\|v\| \leq \left( \frac{1}{4 C } \right)^{1/(2p-2)} =: r$ we have
\[
\cI(v) \leq \frac{1}{4} \|v\|^2.
\]
Hence $\cE(v) \geq \frac{1}{4}\|v\|^2 = \frac{r^2}{4} > 0$ for $\|v\| = r$.
\item[(J2)] Observe that
\begin{multline*}
\cI(t_n v_n) / t_n^q = \frac{\mathscr{D}(t_n v_n)}{t_n^q} - \frac{1}{q} \int_{\R^N} \Gamma(x) |\gamma(v_n)|^q \, dy \\
= t_n^{2p-q} \mathscr{D}(v_n) - \frac{1}{q} \int_{\R^N} \Gamma(x) |\gamma(v_n)|^q \, dy \to \infty.
\end{multline*}
\item[(J3)] Fix $v \in \cN$ and define
\[
\varphi(t) := \frac{t^2-1}{2} \cI'(v)(v) - \cI(tv) + \cI(v)
\]
for any $t \geq 0$. Note that $\varphi(1) = 0$. Moreover
\begin{align*}
\varphi'(t) &= t \cI'(v)(v) - \cI'(tv)(v) = t \left( \frac{1}{2p} \mathscr{D}'(v)(v) - \int_{\R^N} \Gamma(x) |\gamma(v_n)|^q \, dy \right) \\
&{}\quad - \frac{1}{2p} \mathscr{D}'(tv)(v) + t^{q-1} \int_{\R^N} \Gamma(x) |\gamma(v)|^q \, dy \\
&= \left( t - t^{2p-1} \right) \frac{1}{2p} \mathscr{D}'(v)(v) + \left( t^{q-1} - t \right) \int_{\R^N} \Ga (x) |\gamma (v)|^q \, dy.
\end{align*}
Since $v \in \cN$, we have
\[
\int_{\R^N} \Gamma(y) |\gamma(v)|^q \, dy < \frac{1}{2p} \mathscr{D}'(v)(v).
\]
For $t > 1$ we have $t^{q-1} - t > 0$ and
\[
\varphi'(t) < \left( t^{q-1} - t^{2p-1} \right) \frac{1}{2p} \mathscr{D}'(v)(v) < 0.
\]
Similarly $\varphi'(t) > 0$ for $t < 1$.
\item[(J4)] Suppose that $\{v_n\}_n \subset \cN$ and $\|v_n\| \to \infty$. Note that
\begin{align*}
\cE(v_n) &= \frac{1}{2} \|v_n\|^2 - \frac{1}{2p} \mathscr{D}(v_n) + \frac{1}{q} \int_{\R^N} \Gamma(y) |\gamma(v_n)|^q \, dy =
\\ &= \left( \frac{1}{2} - \frac{1}{q} \right) \|v_n\|^2 + \left( \frac{1}{q} - \frac{1}{2p} \right) \mathscr{D}(v_n) \geq \left( \frac{1}{2} - \frac{1}{q} \right) \|v_n\|^2 \to \infty.
\end{align*}
\end{enumerate}
\end{proof}

\begin{remark}\label{rem-nehari}
From the proof of \cite[Theorem 2.1]{BieganowskiMederski} there holds that for every $v \in H \setminus \{0\}$ there is unique $t = t(v) > 0$ such that $tv \in \cN$. Moreover $t(v) > 0$ is the unique maximum of
$$
[0,\infty) \ni \tau \mapsto \cE (\tau v) \in \R,
$$
in particular $\sup_{\tau \geq 0} \cE(\tau v) = \cE(t(v)v)$.
\end{remark}

%\section{Existence of solutions via the Nehari manifold method}

Define the ground state level
\[
c := \inf_\cN \cE.
\]
In view of Theorem \ref{ThSetting} we have $c > 0$ and there is bounded minimizing sequence $\{v_n\}_n \subset \cN$, i.e.
\[
\cE(v_n) \to c, \quad \cE'(v_n) \to 0.
\]
Since $\mathcal{N}$ contains any critical point of $\mathcal{E}$, it suffices to prove that the value $c$ is attained.
Without loss of generality we may assume that $v_n \weakto v_0$ in $H$ and that
\(
\cE'(v_0) = 0.
\)
Indeed, the sequence $\{v_n\}_n$ is easily seen to be bounded in $H$, so that it converges --- up to a subsequence --- to some $v_0 \in H$.
Fix any $\varphi \in C_0^\infty (\R^{N+1}_+)$: then
\begin{align*}
\cE'(v_n)(\varphi) = \langle v_n, \varphi \rangle - \frac{1}{2p} \mathscr{D}'(v_n)(\varphi) +  \int_{\R^N} \Ga (y) |\gamma(v_n)|^{q-2} \gamma(v_n) \gamma(\varphi) \, dy.
\end{align*}
In view of the weak convergence, we have $\langle v_n, \varphi \rangle \to \langle v_0,\varphi \rangle$.
Recall that $\gamma(\varphi) \in C_0^\infty (\R^N)$. Moreover, for any measurable set $E \subset \supp (\gamma(\varphi))$ we obtain
\[
\int_E | \Gamma (y) | |\gamma(v_n)|^{q-1} |\gamma(\varphi)| \, dy \leq |\Gamma|_\infty |\gamma(v_n)|_{q}^{q-1} |\gamma(\varphi) \chi_E |_q.
\]
Hence, in view of Vitali convergence theorem we have
\[
\int_{\R^N} \Ga (y) |\gamma(v_n)|^{q-2} \gamma(v_n) \gamma(\varphi) \, dy \to \int_{\R^N} \Ga (y) |\gamma(v_0)|^{q-2} \gamma(v_0) \gamma(\varphi) \, dy.
\]
%Moreover, for any measurable $E \subset \supp (\gamma(\varphi))$ we have
%\begin{multline*}
%\int_{E} \left| (I_\alpha * |\gamma(v_n)|^p) |\gamma(v_n)|^{p-2} \gamma(v_n) \gamma(\varphi) \right| \, dy \leq |I_\alpha|_t |\gamma(v_n)|^p_{\frac{2tp}{2t-1}} | \gamma(v_n)^{p-1} \gamma(\varphi) \chi_E |_{\frac{2t}{2t-1}} \\
%\leq |I_\alpha|_t  |\gamma(v_n)|^p_{\frac{2tp}{2t-1}} | \gamma(v_n) |_{\frac{2tp}{2t-1}}^{p-1} |\gamma(\varphi) \chi_E|_{\frac{2tp}{2t-1}}^{p-1},
%\end{multline*}
%where $t$ is such that \eqref{1.8} holds. Thus the family $\left\{ (I_\alpha * |\gamma(v_n)|^p) |\gamma(v_n)|^{p-2} \gamma(v_n) \gamma(\varphi) \right\}_{n\geq 1}$ is uniformly integrable and in view of the Vitali convergence theorem
From Lemma \ref{weakContDprime} we obtain
\begin{align*}
\mathscr{D}'(v_n)(\varphi) \to \mathscr{D}'(v_0)(\varphi).
\end{align*}

Hence $\cE'(v_n)(\varphi) \to \cE'(v_0)(\varphi) = 0$ and we need to check that $v_0 \neq 0$.

\begin{lemma}\label{lem3.3}
Suppose that there exists~$r > 0$ such that
\[
\lim_{n \to +\infty}\sup_{z \in \R_+^{N+1}} \iint_{B(z,r)} |v_n-v_0|^2 \, dx \, dy  = 0.
\]
Then $\{v_n\}_n$ converges strongly to $v_0$ in $H$ and $v_0 \neq 0$. In particular,  $v_0$ is a ground state for $\cE$.
\end{lemma}
\begin{proof}
From Lions' lemma \cite[Lemma I.1]{lions} we obtain that
\[
v_n \to v_0 \ \mathrm{in} \ L^{2(t-1)} (\R^{N+1}_{+}), \quad \gamma(v_n) \to \gamma(v_0) \ \mathrm{in} \ L^{t} (\R^N) \quad \ \mathrm{for} \ \mathrm{any} \ t \in \left(2, \frac{2N}{N-1} \right).
\]
Then we can show that $v_n \to v_0$ in $H$. Indeed, let us write
\begin{gather*}
 \cE'(v_n)(v_n - v_0) = \langle v_n, v_n - v_0 \rangle - \frac{1}{2p} \mathscr{D}'(v_n)(v_n - v_0) + \int_{\R^N} \Ga(y) |\gamma(v_n)|^{q-2} \gamma(v_n) \gamma(v_n-v_0) \, dy \\
= \langle v_n - v_0, v_n - v_0 \rangle + \langle v_0, v_n - v_0 \rangle \\
{}- \frac{1}{2p} \mathscr{D}'(v_n)(v_n - v_0) + \int_{\R^N} \Ga(y) |\gamma(v_n)|^{q-2} \gamma(v_n) \gamma(v_n-v_0) \, dy.
\end{gather*}
Therefore
\begin{multline}\label{3.1}
\| v_n - v_0 \|^2 =
\cE'(v_n)(v_n - v_0) - \langle v_0, v_n - v_0 \rangle + \frac{1}{2p} \mathscr{D}'(v_n)(v_n - v_0) \\
{} - \int_{\R^N} \Ga(y) |\gamma(v_n)|^{q-2} \gamma(v_n) \gamma(v_n-v_0) \, dy.
\end{multline}
Summing up \eqref{3.1} and
\begin{multline*}
0 = \cE'(v_0) ( v_n - v_0) = \\
\langle v_0, v_n - v_0 \rangle - \frac{1}{2p} \mathscr{D}'(v_0)(v_n - v_0) + \int_{\R^N} \Ga(y) |\gamma(v_0)|^{q-2} \gamma(v_0) \gamma(v_n-v_0) \, dy
\end{multline*}
we see that
\begin{align*}
\|v_n-v_0\|^2 &= \cE'(v_n)(v_n - v_0) + \frac{1}{2p} (\mathscr{D}'(v_n) - \mathscr{D}'(v_0))(v_n-v_0) \\ &\quad - \int_{\R^N} \Gamma(y) \left( |\gamma(v_n)|^{q-2} \gamma(v_n) - |\gamma(v_0)|^{q-2} \gamma(v_0) \right) \gamma(v_n-v_0) \, dy.
\end{align*}
Since $\{v_n\}_n$ is a bounded Palais-Smale sequence we have
\begin{equation}\label{3.3}
|\cE'(v_n)(v_n - v_0)| \leq \| \cE'(v_n) \| \|v_n - v_0\| \to 0.
\end{equation}
Taking the H\"older inequality into consideration we get
\begin{align}
\label{3.4}\left| \int_{\R^N} \Ga (y) |\gamma(v_n)|^{q-2} \gamma(v_n) \gamma(v_n-v_0) \, dy \right| \leq |\Ga|_\infty |\gamma(v_n)|_q^{q-1} |\gamma(v_n-v_0)|_q \to 0, \\
\label{3.5}\left| \int_{\R^N} \Ga (y) |\gamma(v_0)|^{q-2} \gamma(v_0) \gamma(v_n-v_0) \, dy \right| \leq |\Ga|_\infty |\gamma(v_0)|_q^{q-1} |\gamma(v_n-v_0)|_q \to 0.
\end{align}
Combining \eqref{3.3}, \eqref{3.4} and \eqref{3.5} we have
\[
\|v_n - v_0 \|^2 = \frac{1}{2p} (\mathscr{D}'(v_n)(v_n-v_0) - \mathscr{D}'(v_0)(v_n-v_0)) + o(1).
\]
Recall the decomposition \eqref{eq:decompI}, and observe that
%\begin{align*}
%\left|\mathscr{D}'(v_n) (v_n-v_0)\right| &= 2p \left| \iint_{\R^N \times \R^N}  \frac{|\gamma(v_n) (y')|^p |\gamma(v_n)(y)|^{p-2} \gamma(v_n)(y) \gamma(v_n-v_0)}{|y-y'|^{N-\alpha}} \, dy' \, dy \right| \\
%&\leq 2p  |I_\alpha|_t | |\gamma(v_n)|^p_{\frac{2tp}{2t-1}} | \gamma(v_n) |_{\frac{2tp}{2t-1}}^{p-1} |\gamma(v_n-v_0)|_{\frac{2tp}{2t-1}}^{p-1} \to 0,
%\end{align*}
\begin{align*}
\allowbreak
| \mathscr{D}'(v_n)(v_n - v_0) | &\leq \int_{\R^N} \left| ( I_\alpha * |\gamma(v_n)|^p ) (y) \right| |\gamma(v_n)(y)|^{p-1} |\gamma(v_n-v_0)(y)| \, dy \\
&\leq \int_{\R^N} \left| ( I_1 * |\gamma(v_n)|^p ) (y) \right| |\gamma(v_n)(y)|^{p-1} |\gamma(v_n-v_0)(y)| \, dy \\
&\quad + \int_{\R^N} \left| ( I_2 * |\gamma(v_n)|^p ) (y) \right| |\gamma(v_n)(y)|^{p-1} |\gamma(v_n-v_0)(y)| \, dy  \\
&\leq \left( C |I_1|_t |\gamma(v_n)|_{\frac{2tp}{2t-1}}^p + |I_2|_\infty |\gamma(v_n)|_p^p \right) \int_{\R^N} |\gamma(v_n)|^{p-1} |\gamma(v_n-v_0)| \, dy \\
&\leq \left( C |I_1|_t |\gamma(v_n)|_{\frac{2tp}{2t-1}}^p + |I_2|_\infty |\gamma(v_n)|_p^p \right) |\gamma (v_n)|_p^{p-1} |\gamma(v_n-v_0)|_p,
\end{align*}
%\todo[inline]{To repair: it should follow directly from the Hardy-Littlewood-Sobolev inequality.}
where $t$ satisfies \eqref{1.8}. Similarly
\[
\left|\mathscr{D}'(v_0) (v_n-v_0)\right| \leq C(p,\alpha,v_0,t) |\gamma(v_n-v_0)|_{p} \to 0.
\]
%\[
%\left|\mathscr{D}'(v_0) (v_n-v_0)\right| \leq C(p,\alpha,v_0,t) |\gamma(v_n-v_0)|_{\frac{2tp}{2t-1}}^{p-1} \to 0.
%\]
Hence $\|v_n - v_0\|^2 \to 0$ and the proof is completed.
\end{proof}

\begin{lemma}\label{lem3.4}
Suppose that there exist $r > 0$, $\alpha > 0$ and a sequence
$(z_n) = (x_n, y_n) \subset \mathbb{Z}^{N+1}_{+}$,
where $\mathbb{Z}^{N+1}_{+} := \mathbb{R}^{N+1}_+ \cap \mathbb{Z}^{N+1}$,
such that
\begin{equation}\label{3.6}
\liminf_{n\to+\infty} \iint_{B(z_n, r)} |v_n-v_0|^2 \, dx \, dy \geq \alpha.
\end{equation}
Then $(z_n)$ is unbounded.
\end{lemma}

\begin{proof}
Suppose by contradiction that $(z_n)$ is bounded. Then $(z_n)$ has a convergent subsequence, that we still denote by the same symbol. In view of the compact embedding of $H^1 (B(z,r))$ into $L^2 (B(z,r))$ we have $v_n \to v_0$ in $L^2_{\mathrm{loc}} (\R^{N+1}_+)$, a contradiction with \eqref{3.6}.
\end{proof}

\section{Profile decomposition of bounded Palais-Smale sequences}\label{sect:4}

We put
\begin{equation}\label{Eper}
\cE_{\mathrm{per}} (v) := \cE(v) - \frac{1}{2} \int_{\R^N} V_{l}(y) \gamma(y)^2 \, dy.
\end{equation}

%\begin{lemma}[{\cite[Lemma 2.4]{Moroz}}]
%Let $N \in \mathbb{N}$, $\alpha \in (0,N)$, $p > \frac{N+\alpha}{2N}$ and $(u_n)$ be a bounded sequence in $L^{ \frac{2Np}{N+\alpha} } (\R^N)$. If $u_n \to u$ almost everywhere on $\R^N$ as $n \to \infty$, then
%$$
%\lim_{n\to\infty} \int_{\R^N} (I_\alpha * |u_n|^p ) |u_n|^p \, dx - \int_{\R^N} (I_\alpha * |u_n - u|^p ) |u_n - u|^p \, dx = \int_{\R^N} (I_\alpha * |u|^p) |u|^p \, dx.
%$$
%\end{lemma}

Let us recall the classical Brezis-Lieb lemma (see eg. \cite[Proposition 4.7.30]{Bogachev}).

\begin{lemma}\label{lem:Brezis-Lieb}
Let $1 \leq p \leq r$, and $\{u_n\}_n \subset  L^r(\R^N)$ be a bounded sequence such that $u_n(x) \to u_0 (x)$ for a.e. $x \in \R^N$. Then
$$
\int_{\R^N} \left| |u_n|^p - |u_n - u_0|^p - |u_0|^p \right|^{r / p} \, dx \to 0 \quad \mathrm{as} \ n \to +\infty.
$$
\end{lemma}

To provide the profile decomposition of bounded Palais-Smale sequences we need the following Brezis-Lieb-type splitting result.

\begin{lemma}[Brezis-Lieb-type lemma for $\mathscr{D}$]\label{lem:brezis-D}
Suppose that $\{v_n\}_n \subset H$ is a bounded sequence such that $v_n(x) \to v_0(x)$ for a.e. $x\in\R^N$. Then
\begin{equation*}
\mathscr{D}(v_n - v_0) - \mathscr{D}(v_n) + \mathscr{D}(v_0)   \to 0 \ \mbox{as} \ n \to +\infty.
\end{equation*}
\end{lemma}

The proof is similar to the proof of \cite[Lemma 2.4]{Moroz} and we include it here for the reader's convenience.

\begin{proof}
Put $u_n := \gamma(v_n) \in H^{1/2}(\R^N)$. See that
\begin{align*}
&\quad \int_{\R^N} ( I_\alpha * |u_n|^p ) |u_n|^p \, dx - \int_{\R^N} (I_\alpha * |u_n - u_0|^p ) |u_n - u_0|^p \, dx \\
&= \int_{\R^N} (I_\alpha * (|u_n|^p - |u_n-u_0|^p)) (|u_n|^p - |u_n-u_0|^p) \, dx \\
&\quad + 2 \int_{\R^N} (I_\alpha * (|u_n|^p - |u_n-u_0|^p)) |u_n-u_0|^p \, dx.
\end{align*}
Let $r := \frac{2Np}{N+\alpha} < \frac{2N}{N-1}$. Obviously $r \geq 2$, since $p \geq 2$. In view of continuous embedding $H^{1/2} (\R^N) \subset L^r (\R^N)$ the sequence $\{u_n\}_n$ is bounded in $L^r (\R^N)$. In view of Lemma \ref{lem:Brezis-Lieb} we have
$$
|u_n|^p - |u_n - u_0|^p \to |u_0|^p \quad \mathrm{in} \ L^{\frac{2N}{N+\alpha}}.
$$
Taking Hardy-Littewood-Sobolev inequality (\cite[Theorem 4.3]{Lieb}) into account we get
$$
I_\alpha * \left( |u_n|^p - |u_n - u_0|^p \right) \to I_\alpha * |u_0|^p \quad \mathrm{in} \ L^{\frac{2N}{N-\alpha}} (\R^N).
$$
But, in view of pointwise convergence and boundedness of $\{|u_n - u_0|^p\}_n$ in $L^{\frac{2N}{N+\alpha}} (\R^N)$, we have that
$$
|u_n - u_0|^p \weakto 0 \quad \mathrm{in} \ L^{\frac{2N}{N+\alpha}} (\R^N).
$$
Therefore
$$
\int_{\R^N} (I_\alpha * (|u_n|^p - |u_n-u_0|^p)) |u_n-u_0|^p \, dx \to 0
$$
and
$$
\mathscr{D}(v_n) - \mathscr{D}(v_n - v_0) \to \mathscr{D}(v_0).
$$
\end{proof}

We provide a splitting-type lemma for bounded Palais-Smale sequences in the spirit of \cite{CotiZelati, JeanjeanTanaka}. The following profile decomposition of Palais-Smale sequences is crucial in the proof of the existence result. The approach is based on Lions' lemma and Brezis-Lieb-type results.

\begin{lemma}[Splitting-type lemma]\label{lem:splitting}
Let $\{v_n\}_n$ be a bounded Palais-Smale sequence. Then (up to a subsequence) there is an integer $\ell \geq 0$ and sequences $(z_n^k) = (x_n^k, y_n^k) \subset \mathbb{Z}^{N+1}_+$, $w^k \in H^1 (\R^{N+1}_+)$, $k =1,\ldots, \ell$ such that
\begin{enumerate}
\item[(i)] $v_n \weakto v_0$ and $\cE'(v_0) = 0$;
\item[(ii)] $|y_n^k| \to \infty$ and $|y_n^k - y_n^{k'}| \to \infty$ for $k \neq k'$;
\item[(iii)] $w^k \neq 0$ and $\cE_{\mathrm{per}} ' (w^k) = 0$ for $1 \leq k \leq \ell$;
\item[(iv)] $\left\| u_n - u_0 - \sum_{k=1}^\ell w^k (\cdot - z_n^k) \right\| \to 0$;
\item[(v)] $\cE(v_n) \to \cE(v_0) + \sum_{k=1}^\ell \cE_{\mathrm{per}} (w^k)$.
\end{enumerate}
\end{lemma}

\begin{proof} $ $ \\
%\begin{enumerate}
\textbf{Step 1:} \textit{(i) holds.} \\
Since $\{v_n\}_n$ is bounded, we may pass to a subsequence and assume that $v_n \weakto v_0$ in $H$.
We will show that that
\(
\cE'(v_0) = 0.
\)
Indeed, fix any $\varphi \in C_0^\infty (\R^{N+1}_+)$. Then
\begin{align*}
\cE'(v_n)(\varphi) = \langle v_n, \varphi \rangle - \frac{1}{2p} \mathscr{D}'(v_n)(\varphi) +  \int_{\R^N} \Ga (y) |\gamma(v_n)|^{q-2} \gamma(v_n) \gamma(\varphi) \, dy.
\end{align*}
In view of the weak convergence, we have $\langle v_n, \varphi \rangle \to \langle v_0,\varphi \rangle$. Recall that $\gamma(\varphi) \in C_0^\infty (\R^N)$. Moreover, for any measurable set $E \subset \supp (\gamma(\varphi))$ we obtain
\[
\int_E | \Gamma (y) | |\gamma(v_n)|^{q-1} |\gamma(\varphi)| \, dy \leq |\Gamma|_\infty |\gamma(v_n)|_{q}^{q-1} |\gamma(\varphi) \chi_E |_q.
\]
Hence, in view of Vitali convergence theorem we have
\[
\int_{\R^N} \Ga (y) |\gamma(v_n)|^{q-2} \gamma(v_n) \gamma(\varphi) \, dy \to \int_{\R^N} \Ga (y) |\gamma(v_0)|^{q-2} \gamma(v_0) \gamma(\varphi) \, dy .
\]
%Moreover, for any measurable $E \subset \supp (\gamma(\varphi))$ we have
%\begin{align*}
%\int_{E} \left| (I_\alpha * |\gamma(v_n)|^p) |\gamma(v_n)|^{p-2} \gamma(v_n) \gamma(\varphi) \right| \, dy &\leq |I_\alpha|_t | |\gamma(v_n)|^p_{\frac{2tp}{2t-1}} | \gamma(v_n)^{p-1} \gamma(\varphi) \chi_E |_{\frac{2t}{2t-1}} \\
%&\leq |I_\alpha|_t | |\gamma(v_n)|^p_{\frac{2tp}{2t-1}} | \gamma(v_n) |_{\frac{2tp}{2t-1}}^{p-1} |\gamma(\varphi) \chi_E|_{\frac{2tp}{2t-1}}^{p-1}
%\end{align*}
%and $t$ satisfies \eqref{1.8}. Thus the family $\left\{ (I_\alpha * |\gamma(v_n)|^p) |\gamma(v_n)|^{p-2} \gamma(v_n) \gamma(\varphi) \right\}_{n \geq 1}$ is uniformly integrable and in view of the Vitali convergence theorem
From Lemma \ref{weakContDprime}
$$
\mathscr{D}'(v_n)(\varphi) \to \mathscr{D}'(v_0)(\varphi)
$$
Hence $\cE'(v_n)(\varphi) \to \cE'(v_0)(\varphi) = 0$. \\ \\
\textbf{Step 2:} \textit{Assume that there exists $r > 0$ such that}
\begin{equation*}
\sup_{z \in \R_+^{N+1}} \iint_{B(z,r)} |v_n-v_0|^2 \, dx \, dy \to 0.
\end{equation*}
\textit{Then (ii)--(v) hold for $\ell = 0$.} \\
In view of Lemma \ref{lem3.3} we have $v_n \to v$ in $H$, $v_0 \neq 0$ and $v_0$ is a ground state for $\cE$. In particular, (ii)--(v) hold for $\ell = 0$. \\ \\
\textbf{Step 3:} \textit{Assume that there are $r > 0$ and $\alpha > 0$ and a sequence $(z_n) = (x_n, y_n) \subset \mathbb{Z}_+^{N+1}$ such that \eqref{3.6} holds. Then there is $w \in H$ such that (up to a subsequence)}
\begin{equation*}
(a) \ |y_n| \to \infty, \quad (b) \ v_n (\cdot - z_n) \weakto w \neq 0, \quad (c) \ \cE_{\mathrm{per}}'(w) = 0.
\end{equation*}
From Lemma \ref{lem3.4} we see that $(z_n)$ is unbounded. Suppose that $|y_n| \to \infty$. Put $w_n (z) := v_n(z-z_n) = v_n(x-x_n,y-y_n)$. See that
\begin{equation*}
\|w_n\| \leq C \| w_n\|_{H^1 (\R^{N+1}_+)} = C \| v_n \|_{H^1 (\R^{N+1}_+)},
\end{equation*}
hence $\{w_n\}_n$ is also bounded in $H$ and $w_n \weakto w$ in $H$, $w_n(z) \to w(z)$ for a.e. $z \in \R^{N+1}_+$. Observe that
\begin{align*}
\alpha &\leq \iint_{B((x_n,y_n),r)} |v_n-v_0|^2 \, dx \, dy = \iint_{B((0,0),r)} |w_n - v_0 (x+x_n,y+y_n)|^2 \, dx \, dy \\
&= \iint_{B((0,0),r)} |w_n|^2 \, dx \, dy - 2 \iint_{B(z_n,r)} v_n v_0 \, dx \, dy + \iint_{B(z_n,r)} |v_0|^2 \, dx \, dy.
\end{align*}
Obviously
\begin{align*}
\iint_{B(z_n,r)} v_n v_0 \, dx \, dy &\to 0, \\
\iint_{B(z_n,r)} |v_0|^2 \, dx \, dy &\to 0
\end{align*}
and therefore
\begin{equation*}
\alpha \leq \iint_{B(z_n,r)} |v_n-v_0|^2 \, dx \, dy \leq \iint_{B(0,r)} |w_n|^2 \, dx \, dy + o(1).
\end{equation*}
Hence $w \neq 0$ on $B((0,0),r)$, in particular $w \neq 0$ in $H$. We will show that $\cE_{\mathrm{per}}'(w) = 0$. Take any test function $\varphi \in C_0^\infty (\R^{N+1}_+)$. We compute
\begin{align*}
o(1) &= \cE' (v_n) (\varphi(\cdot + z_n)) \\ &= \cE' (w_n)(\varphi) - \int_{\R^N}  \phi_{\gamma(v_n)} (y) \gamma(\varphi)(y+y_n) \, dy + \int_{\R^N} \phi_{\gamma(w_n)} (y) \gamma(\varphi)(y) \, dy \\
&= \cE_{\mathrm{per}}' (w_n)(\varphi) + \int_{\R^N} V_l (y) \gamma(v_n) \gamma(\varphi(\cdot + z_n)) \, dy \\ &\quad - \int_{\R^N} \phi_{\gamma(v_n)} (y-y_n) \gamma(\varphi)(y) \, dy + \int_{\R^N}  \phi_{\gamma(w_n)} (y) \gamma(\varphi)(y) \, dy.
\end{align*}
In view of Lemma \ref{phiProperties} (ii) we have
\begin{multline*}
- \int_{\R^N} \phi_{\gamma(v_n)} (y-y_n) \gamma(\varphi)(y) \, dy + \int_{\R^N}  \phi_{\gamma(w_n)} (y) \gamma(\varphi)(y) \, dy \\
= - \int_{\R^N} \phi_{\gamma(w_n)} (y) \gamma(\varphi)(y) \, dy + \int_{\R^N}  \phi_{\gamma(w_n)} (y) \gamma(\varphi)(y) \, dy = 0,
\end{multline*}
so
\begin{align*}
o(1) &= \cE_{\mathrm{per}}' (w_n)(\varphi) + \int_{\R^N} V_l (y) \gamma(v_n) \gamma(\varphi(\cdot + z_n)) \, dy \\
&= \cE_{\mathrm{per}}' (w_n)(\varphi) + \int_{\R^N} V_l (y) \gamma(v_n) \gamma(\varphi)(y + y_n) \, dy \\
&= \cE_{\mathrm{per}}' (w_n)(\varphi) + \int_{\R^N} V_l (y-y_n) \gamma(w_n)(y) \gamma(\varphi)(y) \, dy.
\end{align*}
From Vitali convergence theorem we have
$$
\int_{\R^N} V_l (y-y_n) \gamma(w_n)(y) \gamma(\varphi)(y) \, dy \to 0,
$$
hence
$$
\cE_{\mathrm{per}}' (w_n)(\varphi) \to 0.
$$
Similarly as in Step 1 we show that $\cE_{\mathrm{per}}' (w_n)(\varphi) \to \cE_{\mathrm{per}}' (w)(\varphi)$ and therefore
$$
\cE_{\mathrm{per}}' (w)(\varphi) = 0,
$$
and the proof of Step 3 is completed in this case. Thus we may assume that $(y_n)$ is bounded and $x_n\to\infty$. We will show that this is impossible. We will repeat the arguments from the proof of \cite[Lemma 4.2]{Coti}. Since $(y_n)$ is bounded, it has a convergent subsequence. From \eqref{3.6} it follows that
\begin{gather*}%\label{3.7}
\iint_{B(z_n, r)} |v_n|^2 \, dx \, dy \geq \beta > 0
\end{gather*}
for some $\beta > 0$. We may assume without loss of generality that $r>2$ and that $y_n = 0$. Since $v_0 \in L^2 (\R^N)$ we have
\[
\iint_{B(z_n, r)} |v_0|^2 \, dx \, dy \to 0 \quad \mbox{as} \ n \to +\infty.
\]
Moreover, in view of boundedness of $\{v_n\}_n$ in $L^2(\R^N)$ we have
\begin{align*}
\iint_{B(z_n,r)} v_n v_0 \,dx\,dy \to 0 \quad \mbox{as} \ n \to +\infty
\end{align*}
and
\[
\iint_{B(z_n, r)} |v_n - v_0|^2 \, dx \, dy = \iint_{B(z_n, r)} |v_n|^2 \, dx \, dy + o(1).
\]
For any fixed $n \geq 1$ let $R_n\geq 1$ denotes the smallest integer such that
\[
\iint_{ \{ r + R_n \leq |(x,y)-z_n| \leq r + R_n + 1 \}} |\nabla v_n|^2 +  |v_n|^2 \,dx \, dy < \beta.
\]
While $x_n \to \infty$ and in view of boundedness of $\{v_n\}_n$ we may assume that $x_n$ is large enough that
\[
\{ r + R_n \leq |(x,y)-z_n| \leq r + R_n + 1 \} \subset \mathbb{R}^{N+1}_+,
\]
i.e.
\begin{equation}\label{xn}
x_n > r + M + 1,
\end{equation}
where
\[
R_n \leq \frac{1}{\beta} \iint_{\R^{N+1}_+} |\nabla v_n|^2 + |v_n|^2 \, dx \, dy \leq M.
\]
Define
\[
\psi_n (x,y) = \varphi_{r+R_n} (|(x,y)-z_n|),
\]
where $\varphi_{r+R_n} \in C^\infty_0 (\R^{N+1}_+)$ is such that
\[
\varphi_{r+R_n} (z) = \left\{ \begin{array}{ll}
1, & \ \mathrm{for} \ |z| \leq r+R_n, \\
0, & \ \mathrm{for} \ |z| > r+R_n + 1,
\end{array} \right.
\]
and $|\varphi_{r+R_n} (z)| \leq 1$, $|\nabla \varphi_{r+R_n} (z)| \leq 1$. Obviously $\psi_n v_n \in H^1 (\R^{N+1}_+)$. Moreover, while $\{v_n\}_n$ is a bounded Palais-Smale sequence
\[
\| \cE'(v_n)(\psi_n v_n) \| \leq \| \cE'(v_n) \| \cdot \|\psi_n v_n\| \to 0 \quad \mbox{as} \ n\to +\infty.
\]
Hence
\begin{multline*}
o(1) = \cE'(v_n)(\psi_n v_n) = \\
 \iint_{\R^{N+1}_+} \nabla v_n \nabla (\psi_n v_n) \, dx \, dy + m^2 \iint_{\R^{N+1}_+} |v_n|^2 \psi_n \, dx \, dy + \int_{\R^N} (V(y)-m) \gamma(v_n)\gamma(\psi_n v_n) \, dy \\
{}- \frac{1}{2p} \mathscr{D}'(v_n)(\psi_n v_n) + \int_{\R^N} \Ga(y) |\gamma(v_n)|^{q-2} \gamma(v_n) \gamma(\psi_n v_n) \, dy.
\end{multline*}
In view of the definition of trace and $\psi_n$, and taking \eqref{xn} into account, we have $\gamma(\psi_n v_n) = 0$. Hence
\begin{align*}
o(1) &= \cE'(v_n)(\psi_n v_n) = \iint_{\R^{N+1}_+} \nabla v_n \nabla (\psi_n v_n) \, dx \, dy + m^2 \iint_{\R^{N+1}_+} |v_n|^2 \psi_n \, dx \, dy  \\
&= \iint_{\R^{N+1}_+} |\nabla v_n|^2 \psi_n \, dx \, dy + m^2 \iint_{\R^{N+1}_+} |v_n|^2 \psi_n \, dx \, dy + \iint_{\R^{N+1}_+} v_n \nabla v_n \nabla \psi_n \, dx \, dy \\
&= \iint_{\R^{N+1}_+} (|\nabla v_n|^2 + m^2 |v_n|^2) \psi_n \, dx \, dy + \iint_{\R^{N+1}_+} v_n \nabla v_n \nabla \psi_n \, dx \, dy \\
&\geq \iint_{\R^{N+1}_+} (|\nabla v_n|^2 + m^2 |v_n|^2) \psi_n \, dx \, dy \\ &\quad - \iint_{\{ r + R_n \leq |(x,y)-z_n| \leq r + R_n + 1 \}} | v_n| |\nabla v_n| \, dx \, dy \\
&\geq \iint_{B(z_n,r)} |\nabla v_n|^2 + m^2 |v_n|^2 \, dx \, dy \\
&\quad - \frac{1}{2} \iint_{\{ r + R_n \leq |(x,y)-z_n| \leq r + R_n + 1 \}} |\nabla v_n|^2 + |v_n|^2 \, dx \, dy \\
&\geq \iint_{B(z_n,r)} |\nabla v_n|^2 + m^2 |v_n|^2 \, dx \, dy - \frac{\beta}{2} \geq \beta - \frac{\beta}{2} = \frac{1}{2} \beta > 0,
\end{align*}
a contradiction, which completes the proof of Step 3. \\ \\
\textbf{Step 4:} \textit{Assume that there is $m \geq 1$, $(z_n^k) = (x_n^k, y_n^k) \subset \mathbb{Z}_+^{N+1}$, $w^k \in H$ for $k \in \{1,2,\ldots, m\}$ such that
\begin{align*}
|y_n^k| \to \infty, \ |y_n^k - y_n^{k'}| \to \infty \ &\mbox{for} \ 1 \leq k < k' \leq m; \\
v_n^k (\cdot + z_n^k) \to w^k \neq 0 \ &\mbox{for} \ 1 \leq k \leq m; \\
\cE_{\mathrm{per}}' (w^k) = 0 \ &\mbox{for} \ 1 \leq k \leq m.
\end{align*}
Then
\begin{itemize}
\item[(1)] if there is $r > 0$ such that
\begin{equation}\label{3.8}
\sup_{z \in \R_+^{N+1}} \iint_{B(z,r)} \left| v_n - v_0 - \sum_{k=1}^m w^k(\cdot - z_n^k) \right|^2 \, dx \, dy \to 0 \ \mbox{as} \ n \to +\infty,
\end{equation}
then
$$
\left\| v_n - v_0 - \sum_{k=1}^m w^k(\cdot - z_n^k) \right\| \to 0;
$$
\item[(2)] if there is $r > 0$ and $(z_n^{m+1}) = (x_n^{m+1},y_n^{m+1}) \subset \mathbb{Z}_+^{N+1}$ such that
\begin{equation}\label{3.9}
\liminf_{n\to\infty}  \iint_{B(z_n^{m+1}, r)} \left| v_n - v_0 - \sum_{k=1}^m w^k(\cdot - z_n^k) \right|^2 \, dx \, dy > 0
\end{equation}
then there is $w^{m+1} \in H \setminus \{0\}$ such that (up to subsequences)
\begin{itemize}
\item[(a)]  $|y_n^{m+1}| \to \infty$, $|y_n^{m+1} - y_n^k| \to \infty$ for $1 \leq k \leq m$,
\item[(b)] $v_n (\cdot - z_n^{m+1}) \weakto w^{m+1}$,
\item[(c)] $\cE_{\mathrm{per}}'(w^{m+1}) = 0$.
\end{itemize}
\end{itemize}
}
Suppose that \eqref{3.8} holds and put
\[
\xi_n := v_n - v_0 - \sum_{k=1}^m w^k(\cdot - z_n^k).
\]
From Lion's lemma we have
\[
\xi_n \to 0 \ \mbox{in} \ L^{2(t-1)} (\R_+^{N+1}) \ \mbox{and} \ \gamma(\xi_n) \to 0 \ \mbox{in} \ L^t (\R^N) \quad \mbox{for} \ t \in \left(2, \frac{2N}{N-1} \right).
\]
We denote that
\[
\cE'(v_n)(\xi_n) = \langle v_n, \xi_n \rangle - \frac{1}{2p} \mathscr{D}'(v_n)(\xi_n) + \int_{\R^N} \Gamma(y) |\gamma(v_n)|^{q-2} \gamma(v_n) \gamma(\xi_n) \, dy.
\]
Obviously $|\cE'(v_n)(\xi_n)| \leq \| \cE'(v_n) \| \| \xi_n \| \to 0$ and
\[
\int_{\R^N} \Gamma(y) |\gamma(v_n)|^{q-2} \gamma(v_n) \gamma(\xi_n) \, dy \to 0,
\]
so
\begin{multline*}
o(1) = \langle v_n, \xi_n \rangle - \frac{1}{2p} \mathscr{D}'(v_n)(\xi_n) = \| \xi_n \|^2 + \langle v_0, \xi_n \rangle + \sum_{k=1}^m \langle w^k(\cdot - z_n^k), \xi_n \rangle - \frac{1}{2p} \mathscr{D}'(v_n)(\xi_n).
\end{multline*}
Moreover
\begin{equation}\label{3.10}
0 = \cE'(v_0)(\xi_n) = \langle v_0, \xi_n \rangle - \frac{1}{2p} \mathscr{D}'(v_0)(\xi_n) + o(1),
\end{equation}
while
\[
\int_{\R^N} \Gamma(y) |\gamma(v_0)|^{q-2} \gamma(v_0) \gamma(\xi_n) \, dy \to 0.
\]
Using \eqref{3.10} we obtain
\begin{equation}\label{3.11}
\| \xi_n \|^2 = - \frac{1}{2p} \mathscr{D}'(v_0)(\xi_n) - \sum_{k=1}^m \langle w^k(\cdot - z_n^k), \xi_n \rangle + \frac{1}{2p} \mathscr{D}'(v_n)(\xi_n) + o(1).
\end{equation}
Recall that $\cE_{\mathrm{per}}'(w^k) = 0$. Hence
\begin{align*}
0 &= \cE_{\mathrm{per}}'(w^k)(\xi_n(\cdot + z_n^k))= \langle w^k, \xi_n(\cdot + z_n^k) \rangle - \frac{1}{2p} \mathscr{D}'(w^k)(\xi_n (\cdot + z_n^k)) \\ &\quad + \int_{\R^N} \Ga(y) | \gamma(w^k)|^{q-2} \gamma(w^k) \gamma(\xi_n(\cdot + z_n^k)) \, dy - \frac{1}{2} \int_{\R^N} V_l(y) \gamma(w^k) \gamma(\xi_n(\cdot + z_n^k)) \, dy \\
&= \langle w^k (\cdot - z_n^k), \xi_n \rangle - \frac{1}{2p} \mathscr{D}'(w^k)(\xi_n (\cdot + z_n^k)) \\ &\quad + \int_{\R^N} \Ga(y) | \gamma(w^k)|^{q-2} \gamma(w^k) \gamma(\xi_n(\cdot + z_n^k)) \, dy - \frac{1}{2} \int_{\R^N} V_l(y) \gamma(w^k(\cdot - z_n^k)) \gamma(\xi_n) \, dy
\end{align*}
Combining it with \eqref{3.11} we have
\begin{multline*}
\| \xi_n\|^2 = - \frac{1}{2p} \sum_{k=1}^m \mathscr{D}'(w^k)(\xi_n (\cdot + z_n^k)) + \sum_{k=1}^m \int_{\R^N} \Ga(y) | \gamma(w^k)|^{q-2} \gamma(w^k) \gamma(\xi_n(\cdot + z_n^k)) \, dy \\ {}- \frac{1}{2}  \sum_{k=1}^m \int_{\R^N} V_l(y) \gamma(w^k(\cdot - z_n^k)) \gamma(\xi_n) \, dy + \frac{1}{2p} [ \mathscr{D}'(v_n)(\xi_n) - \mathscr{D}'(v_0)(\xi_n)] + o(1).
\end{multline*}
Note that
$$
\left| \int_{\R^N} \Ga(y) | \gamma(w^k)|^{q-2} \gamma(w^k) \gamma(\xi_n(\cdot + z_n^k)) \, dy \right| \leq |\Ga|_\infty \int_{\R^N} | \gamma(w^k(\cdot - z_n^k))|^{q-1} | \gamma(\xi_n) | \, dy \to 0.
$$
Thus
\begin{align*}
\| \xi_n\|^2 &= - \frac{1}{2p} \sum_{k=1}^m \mathscr{D}'(w^k)(\xi_n (\cdot + z_n^k)) - \frac{1}{2}  \sum_{k=1}^m \int_{\R^N} V_l(y) \gamma(w^k(\cdot - z_n^k)) \gamma(\xi_n) \, dy \\
&\qquad {}+ \frac{1}{2p} [ \mathscr{D}'(v_n)(\xi_n) - \mathscr{D}'(v_0)(\xi_n)] + o(1).
\end{align*}
From the Vitali convergence theorem we have
\begin{equation*}
\int_{\R^N} V_l(y) \gamma(w^k(\cdot - z_n^k)) \gamma(\xi_n) \, dy \to 0,
\end{equation*}
since $V_l(y + z_n^k) \to 0$ as $n \to +\infty$. Hence
\begin{equation*}
\| \xi_n\|^2 = - \frac{1}{2p} \sum_{k=1}^m \mathscr{D}'(w^k)(\xi_n (\cdot + z_n^k)) + \frac{1}{2p} [ \mathscr{D}'(v_n)(\xi_n) - \mathscr{D}'(v_0)(\xi_n)] + o(1).
\end{equation*}
Moreover
\begin{align*}
\mathscr{D}'(w^k)(\xi_n (\cdot + z_n^k)) &\to 0, \\
\mathscr{D}'(v_n)(\xi_n) &\to 0, \\
\mathscr{D}'(v_0)(\xi_n) &\to 0
\end{align*}
exactly as in Step 1. Thus $\xi_n \to 0$ in $H$.

Suppose now that \eqref{3.9} holds. Exactly as in Step 3 we obtain that $(z_n^{m+1}) = (x_n^{m+1}, y_n^{m+1})$ is unbounded and that $v_n (\cdot - z_n^{m+1}) \weakto w^{m+1}$ for some $w^{m+1} \in H \setminus \{0\}$. If $|y_n^{m+1}| \to \infty$, similarly as in Step 1 we show that
\begin{equation*}
\cE_{\mathrm{per}}'(v_n (\cdot - z_n^{m+1}))(\varphi) \to \cE_{\mathrm{per}}' (w^{m+1})(\varphi)
\end{equation*}
and on the other hand that $\cE_{\mathrm{per}}'(v_n (\cdot - z_n^{m+1}))(\varphi) \to 0$ and the proof is completed in this case. Hence we assume that $(y_n^{m+1})$ is bounded and therefore $x_n^{m+1} \to \infty$, however this cannot hold as in Step 3. \\ \\
\textbf{Step 5:} \textit{Conclusion.} \\
Directly from Step 1 we obtain (i). If
\begin{equation*}
\sup_{z \in \R_+^{N+1}} \iint_{B(z,r)} |v_n-v_0|^2 \, dx \, dy \to 0
\end{equation*}
then from Step 2 the lemma holds for $\ell = 0$. Otherwise, \eqref{3.6} holds and in view of Step 3 there is $w \in H$ and $(z_n) = (x_n, y_n)$ such that (a), (b), (c) hold true. Put $z_n^1 := z_n$ and $w^1 := w$. Then we iterate Step 4. Observe that, from the properties of the weak convergence
\begin{equation*}
0 \leq \lim_{n \to +\infty} \left\| v_n - v_0 - \sum_{k=1}^m w^k (\cdot-z_n^k) \right\|^2 = \lim_{n\to +\infty} \left( \|v_n\|^2 - \| v_0\|^2 - \sum_{k=1}^m \| w^k \|^2 \right).
\end{equation*}
Since $w^k$ are nontrivial critical points of $\cE_{\mathrm{per}}$, we find a number $\rho > 0$ such that $\|w^k\| \geq \rho$. Hence the procedure will finish after a finite number of steps, say $\ell$ steps. Thus we have proven (i)--(iv). \\ \\
\textbf{Step 6:} \textit{(v) holds.} \\
Observe that
\begin{align*}
\cE (v_n) &= \frac{1}{2} \langle v_n, v_n \rangle - \frac{1}{2p} \mathscr{D}(v_n) + \frac{1}{q} \int_{\R^N} \Ga (y) |\gamma(v_n)|^{q} \, dy \\
&= \frac{1}{2} \langle v_0, v_0 \rangle  + \frac{1}{2} \langle v_n - v_0, v_n - v_0 \rangle + \langle v_0, v_n - v_0 \rangle - \frac{1}{2p} \mathscr{D}(v_n) + \frac{1}{q} \int_{\R^N} \Ga (y) |\gamma(v_n)|^{q} \, dy \\
&= \cE (v_0) + \frac{1}{2} \langle v_n - v_0, v_n - v_0 \rangle + \langle v_0, v_n - v_0 \rangle \\
&\quad{}- \frac{1}{2p} \mathscr{D}(v_n) + \frac{1}{q} \int_{\R^N} \Ga (y) |\gamma(v_n)|^{q} \, dy + \frac{1}{2p} \mathscr{D}(v_0) - \frac{1}{q} \int_{\R^N} \Ga (y) |\gamma(v_0)|^{q} \, dy \\
&= \cE (v_0) + \cE_{\mathrm{per}} (v_n - v_0) + \langle v_0, v_n - v_0 \rangle \\
&\quad{}- \frac{1}{2p} \mathscr{D}(v_n) + \frac{1}{q} \int_{\R^N} \Ga (y) |\gamma(v_n)|^{q} \, dy + \frac{1}{2p} \mathscr{D}(v_0) - \frac{1}{q} \int_{\R^N} \Ga (y) |\gamma(v_0)|^{q} \, dy \\
&\quad{}+ \frac{1}{2p} \mathscr{D}(v_n-v_0) - \frac{1}{q} \int_{\R^N} \Ga (y) |\gamma(v_n - v_0)|^{q} \, dy + \frac{1}{2} \int_{\R^N} V_l (y) |\gamma(v_n - v_0)|^2 \, dy.
\end{align*}
In view of the weak convergence we have
$\langle v_0, v_n - v_0 \rangle \to 0$.
Let $E \subset \R^N$ be a measurable set. From the H\"older inequality and \eqref{eq:1.5V} we obtain
\begin{equation*}
\left| \int_{E} V_l (y) |\gamma(v_n - v_0)|^2 \, dy \right| \leq |V_l \chi_E|_s |\gamma(v_n-v_0)|_{\frac{2s}{s-1}}^{2}.
\end{equation*}
Since $2\leq 2s/(s-1)\leq 2N/(N-1)$ and $\{\gamma(v_n-v_0)\}_n$ is bounded in $H^{1/2} (\R^N)$, we obtain that the family $\{ V_l (y) |\gamma(v_n - v_0)|^2 \}_n$ is uniformly integrable and tight. Hence, in view of Vitali convergence theorem
\begin{gather*}
\frac{1}{2} \int_{\R^N} V_l (y) |\gamma(v_n - v_0)|^2 \, dy \to 0.
\end{gather*}
In view of the Bezis-Lieb lemma (\cite[Proposition 4.7.30]{Bogachev}) we easily get
\begin{multline*}
\left| \int_{\R^N} \Ga (y) |\gamma(v_n)|^{q} \, dy - \int_{\R^N} \Ga (y) |\gamma(v_0)|^{q} \, dy -  \int_{\R^N} \Ga (y) |\gamma(v_n - v_0)|^{q} \, dy \right| \\
\quad\leq |\Ga|_\infty \int_{\R^N} \left| |\gamma(v_n)|^{q} - |\gamma(v_0)|^{q} - |\gamma(v_n - v_0)|^{q} \right| \, dy \to 0.
\end{multline*}
Hence, it is enough to show that
\begin{align}
&\mathscr{D}(v_n) - \mathscr{D}(v_0) - \mathscr{D}(v_n - v_0) \to 0, \label{3.12} \\
&\cE_{\mathrm{per}} (v_n - v_0) \to \sum_{k=1}^\ell \cE_{\mathrm{per}}  (w^k). \label{3.13}
\end{align}
The convergence \eqref{3.12} follows by Lemma \ref{lem:brezis-D}. We will focus on \eqref{3.13}. We compute that
\begin{align*}
\cE_{\mathrm{per}} (v_n - v_0) &= \frac{1}{2} \| v_n - v_0 \|^2 - \frac{1}{2p} \mathscr{D}(v_n - v_0) \\
&\quad + \frac{1}{q} \int_{\R^N} \Gamma(y) | \gamma(v-v_0) |^q \, dy - \frac{1}{2} \int_{\R^N} V_{loc}(y) | \gamma(v-v_0) |^2 \, dy \\
&= \frac{1}{2} \left\| v_n - v_0 - \sum_{k=1}^\ell w^k (\cdot-z_n^k) \right\|^2 - \frac{1}{2p} \mathscr{D}(v_n - v_0) \\
&\quad + \frac{1}{q} \int_{\R^N} \Gamma(y) | \gamma(v-v_0) |^q \, dy - \int_{\R^N} \frac{1}{2} V_{loc}(y) | \gamma(v-v_0) |^2 \, dy \\
&\quad + \sum_{k=1}^\ell \| w^k (\cdot - z_n^k) \|^2 + o(1) \\
&= \sum_{k=1}^\ell \cE_{\mathrm{per}} (w^k) + \frac{1}{2p} \sum_{k=1}^\ell \mathscr{D}(w^k (\cdot - z_n^k)) - \frac{1}{q} \sum_{k=1}^\ell \int_{\R^N} \Ga (y) | \gamma(w^k(\cdot - z_n^k)) |^q \, dy \\
&\quad -\frac{1}{2p} \mathscr{D}(v_n - v_0) + \frac{1}{q} \int_{\R^N} \Ga (y) | \gamma(v-v_0)|^q \, dy + o(1).
\end{align*}
Iterating Lemma \ref{lem:brezis-D} we get
\begin{equation*}
\mathscr{D}(v_n-v_0) - \sum_{k=1}^\ell \mathscr{D}(w^k(\cdot - z_n^k)) \to 0.
\end{equation*}
Similarly, iterating the Brezis-Lieb lemma we get
\begin{equation*}
\int_{\R^N}  \Ga (y) | \gamma(v-v_0)|^q \, dy - \sum_{k=1}^\ell \int_{\R^N} \Ga (y) | \gamma(w^k(\cdot - z_n^k)) |^q \, dy \to 0.
\end{equation*}
All details can be found in \cite{BieganowskiMederski}. Hence
\begin{equation*}
\cE_{\mathrm{per}}(v_n - v_0) = \sum_{k=1}^\ell \cE_{\mathrm{per}} (w^k) + o (1)
\end{equation*}
and the proof is finished.
%\todo[inline]{We need to show (v).}
%\end{enumerate}
\end{proof}
%

%
%\begin{lemma}\label{lem:c_per}
%There holds
%\begin{equation*}
%c < c_{\mathrm{per}}.
%\end{equation*}
%\end{lemma}
%
%\begin{proof}
%From Section \ref{subsect:periodic} we already know that there is $v_{\mathrm{per}} \in \cN_{\mathrm{per}}$ such that
%\begin{equation*}
%c_{\mathrm{per}} = \cE_{\mathrm{per}} ( v_{\mathrm{per}} ) > 0.
%\end{equation*}
%Let $t_{\mathrm{per}} > 0$ be a number such that $t_{\mathrm{per}} v_{\mathrm{per}} \in \cN$. The inequality $V_l < 0$ implies that $V(y) < V_p (y)$ for a.e. $y \in \R^N$. Therefore
%\begin{equation*}
%c_{\mathrm{per}} = \cE_{\mathrm{per}} ( v_{\mathrm{per}} ) \geq \cE_{\mathrm{per}} ( t_{\mathrm{per}} v_{\mathrm{per}} ) > \cE ( t_{\mathrm{per}} v_{\mathrm{per}} ) \geq \inf_{\cN} \cE = c > 0
%\end{equation*}
%and the proof is completed.
%\end{proof}
%
\section{Proof of Theorem \ref{th:main}}\label{sect:5}

Define $c_{\mathrm{per}} := \inf_{\cN_{\mathrm{per}}} \cE_{\mathrm{per}} > 0$, where
\begin{equation}\label{Nper}
\cN_{\mathrm{per}} := \{ v \in H \setminus \{0\} \mid \cE_{\mathrm{per}}'(v)(v) = 0 \}.
\end{equation}

\begin{altproof}{Theorem \ref{th:main}}
From Lemma \ref{lem:splitting} (v) we obtain
\begin{equation}\label{q}
c = \cE(v_0) + \sum_{k=1}^\ell \cE_{\mathrm{per}} (w^k) \geq \cE(v_0) + \ell c_{\mathrm{per}}.
\end{equation}
\begin{enumerate}
\item[(a)] Assume that $V_l = 0$. Then $\cE = \cE_{\mathrm{per}}$, $c = c_{\mathrm{per}}$ and from \eqref{q} there holds
$$
c \geq \cE (v_0) + \ell c.
$$
If $v_0 \neq 0$, we get $c \geq (\ell + 1) c$, $\ell = 0$ and $v_0$ is a ground state. If $v_0 = 0$ we have $c \geq \ell c$. Since $c > 0$ then $\ell = 1$ and $w^1 \neq 0$ is a ground state.

\item[(b)] Assume that $V_l(y) < 0$ for a.e. $y \in \R^N$. Suppose that $v_0 = 0$ and $\cE(v_0) = \cE(0) = 0$. From (a) we already know that there is $v_{\mathrm{per}} \in \cN_{\mathrm{per}}$ such that
\begin{equation*}
c_{\mathrm{per}} = \cE_{\mathrm{per}} ( v_{\mathrm{per}} ) > 0.
\end{equation*}
Let $t_{\mathrm{per}} > 0$ be a number such that $t_{\mathrm{per}} v_{\mathrm{per}} \in \cN$. The inequality $V(y) < V_p (y)$ for a.e. $y \in \R^N$ implies that
\begin{equation}\label{cper}
c_{\mathrm{per}} = \cE_{\mathrm{per}} ( v_{\mathrm{per}} ) \geq \cE_{\mathrm{per}} ( t_{\mathrm{per}} v_{\mathrm{per}} ) > \cE ( t_{\mathrm{per}} v_{\mathrm{per}} ) \geq \inf_{\cN} \cE = c > 0.
\end{equation}
From \eqref{q} and \eqref{cper} we get the inequality
\begin{equation*}
c > \ell c.
\end{equation*}
Therefore $\ell = 0$ and $c = \cE(0) = 0$, a contradiction. Hence $v_0 \neq 0$ is a ground state solution.
\end{enumerate}
The regularity of $u$ can be shown by adapting \cite[Lemma 5.1 and Theorem 7.1]{Cingolani}.
%\todo[inline]{Does the regularity hold if we do not know that $u>0$?}
\end{altproof}

\section{Proof of Theorem \ref{th:main2}}\label{sect:6}

%\begin{altproof}{Theorem \ref{th:main2}}
%\todo[inline,color=green]{It looks fine, I have only replaced $|t_z|$ with $t_z$, since $t_z>0$}
Let $\cE_{\mathrm{per}}$ and $\cN_{\mathrm{per}}$ be given by \eqref{Eper} and \eqref{Nper} respectively. Suppose by contradiction that there is a ground state $v_0 \in \cN$ of $\cE$. In view of Remark \ref{rem-nehari} there is $t_{\mathrm{per}} > 0$ such that $t_{\mathrm{per}}v_0 \in \cN_{\mathrm{per}}$. Since $v_0 \neq 0$ and $V_l (y) > 0$ for a.e. $y \in \R^N$ we have
\begin{equation*}
\int_{\R^N} V_l (y) |\gamma(v_0)|^2 \, dy > 0
\end{equation*}
and therefore
\begin{equation*}
c_{\mathrm{per}} := \inf_{ \cN_{\mathrm{per}} } \cE_{\mathrm{per}} \leq \cE_{\mathrm{per}} ( t_{\mathrm{per}}v_0 ) < \cE ( t_{\mathrm{per}}v_0) \leq \cE(v_0) = \inf_{\cN} \cE =: c.
\end{equation*}
On the other hand, take any $v \in \cN_{\mathrm{per}}$ and for any $z \in \mathbb{Z}^{N}$ let $v_z (x,y) := v(x, y - z)$. Let $t_z > 0$ be a number such that $t_z v_z \in \cN$. Then
\begin{multline*}
\cE_{\mathrm{per}}(v) = \cE_{\mathrm{per}}(v_z) \geq \cE_{\mathrm{per}}(t_z v_z) \\
= \cE (t_z v_z) - \int_{\R^N} V_l(y) | \gamma (t_z v_z) |^2 \, dy \geq c - \int_{\R^N} V_l(y) | \gamma (t_z v_z) |^2 \, dy.
\end{multline*}
Observe that
\begin{equation*}
\int_{\R^N} V_l(y) | \gamma (t_z v_z) |^2 \, dy = t_z^2 \int_{\R^N} V_l(y) | \gamma (v_z) |^2 \, dy = t_z^2 \int_{\R^N} V_l(y+z) | \gamma (v) |^2 \, dy.
\end{equation*}
In view of (V3) we have by Lebesgue's Theorem
\begin{equation*}
\int_{\R^N} V_l(y+z) | \gamma (v) |^2 \, dy \to 0 \quad \mathrm{as} \ |z|\to\infty.
\end{equation*}
The functional $\mathcal{E}_{\mathrm{per}}$ being coercive on $\mathcal{N}_{\mathrm{per}}$, $\cE_{\mathrm{per}}(t_z v_z) = \cE_{\mathrm{per}}(t_z v) \leq c_{\mathrm{per}}$ implies that  $\sup_{z \in \mathbb{Z}^N} t_z < +\infty$. Hence
\begin{equation*}
\cE_{\mathrm{per}}(v) \geq c - \int_{\R^N} V_l(y) | \gamma (t_z v_z) |^2 \, dy = c - t_z^2 \int_{\R^N} V_l(y+z) | \gamma (v) |^2 \, dy \to c.
\end{equation*}
Taking infimum over all $v \in \cN_{\mathrm{per}}$ we obtain $c_{\mathrm{per}} \geq c$, a contradiction. This concludes the proof.
%\end{altproof}

\section{Proof of Theorem \ref{th:main3}}\label{sect:7}

Suppose that $\{ \Gamma_n \} \subset L^\infty (\R^N)$ is a sequence such that $\Gamma_n$ satisfies ($\Ga$) and $\Gamma_n \to 0$ in $L^\infty (\R^N)$ as $n \to +\infty$. Let $\cE_n$ denotes the Euler functional for $\Gamma = \Gamma_n$.

\begin{lemma}\label{lem:7.1}
There exists a positive radius $r > 0$ such that
\begin{gather*}
\inf_{n \geq 1} \inf_{\|v\| = r} \cE_n (u) > 0.
\end{gather*}
\end{lemma}

\begin{proof}
In view of \eqref{eq2}
\[
\frac{1}{2p} \mathscr{D}(v) - \frac{1}{q} \int_{\R^N} \Gamma_n (y) | \gamma(v) |^q \, dy \leq \frac{1}{2p} \mathscr{D}(v) \leq C \| v\|^{2p} = C \|v\|^{2p-2} \|v\|^2.
\]
For $v \in H$ such that $\|v\| \leq \left( \frac{1}{4 C } \right)^{1/(2p-2)} =: r$ we have
\[
\frac{1}{2p} \mathscr{D}(v) - \frac{1}{q} \int_{\R^N} \Gamma_n (y) | \gamma(v) |^q \, dy  \leq \frac{1}{4} \|v\|^2.
\]
Hence $\cE_n(v) \geq \frac{1}{4}\|v\|^2 = \frac{r^2}{4} > 0$ for $\|v\| = r$.
\end{proof}

Recall that for any $n \geq 1$ there is $v_n \in H$ such that $\cE_n ' (v_n) = 0$ and $\cE_n (v_n) = \inf_{\cN_n} \cE_n$, where $\cN_n$ is the corresponding Nehari manifold. By $\cE_0$ and $\cN_0$ we denote the Euler functional and the Nehari manifold for $\Gamma \equiv 0$. Recall that there is also a ground state $v_0 \in H$ for $\cE_0$. We define
\begin{gather*}
c_n := \cE_n (v_n) = \inf_{\cN_n} \cE_n, \quad c_0 := \cE_0 (v_0) = \inf_{\cN_0} \cE_0.
\end{gather*}

\begin{lemma}\label{lem:7.2}
There holds
\begin{equation*}
\lim_{n \to \infty} c_n = c_0.
\end{equation*}
\end{lemma}

\begin{proof}
Take $t_n > 0$ such that $t_n v_n \in \cN_0$ and note that
\begin{equation}\label{7-1}
c_n \geq \cE_n (t_n v_n) = \cE_0 (t_n v_n) + \frac{t_n^q}{q} \int_{\R^N} \Ga_n(y)|\gamma(v_n)|^q \, dx  \geq c_0 + \frac{t_n^q}{q} \int_{\R^N} \Ga_n(y)|\gamma(v_n)|^q \, dx.
\end{equation}
Similarly, taking $s_n > 0$ such that $s_n v_0 \in \cN_n$ we get
\begin{equation}\label{7-2}
c_0 \geq \cE_0 (s_n v_0) = \cE_n (s_n v_0) - \frac{s_n^q}{q} \int_{\R^N} \Ga_n(y) | \gamma(v_0)|^q \, dy \geq c_n - \frac{s_n^q}{q} \int_{\R^N} \Ga_n(y) | \gamma(v_0)|^q \, dy.
\end{equation}
Combining \eqref{7-1} and \eqref{7-2} we arrive at
\begin{gather*}
c_0 \leq c_0 + \frac{t_n^q}{q} \int_{\R^N} \Ga_n(y)|\gamma(v_n)|^q \, dx \leq c_n \leq c_0 + \frac{s_n^q}{q} \int_{\R^N} \Ga_n(y) | \gamma(v_0)|^q \, dy.
\end{gather*}
Since obviously
\begin{gather*}
\int_{\R^N} \Ga_n(y) | \gamma(v_0)|^q \, dy \leq |\Ga_n|_\infty \int_{\R^N} | \gamma(v_0)|^q \, dy \to 0,
\end{gather*}
it is sufficient to show that $(s_n)$ is bounded. Suppose by contradiction that $s_n \to \infty$. Taking into account that $s_n v_0 \in \cN_n$ we have
\begin{gather*}
s_n^2 \|v_0\|^2 - \mathscr{D}'(s_n v_0) (s_n v_0) + s_n^q \int_{\R^N} \Ga_n(y) |\gamma(v_0)|^q \, dy = 0.
\end{gather*}
Hence, recalling that $q<2p$,
\begin{align*}
0 &= \frac{\|v_0\|^2}{s_n^{q-2}} - \frac{\mathscr{D}'(s_n v_0) (s_n v_0)}{s_n^q} + \int_{\R^N} \Ga_n(y) |\gamma(v_0)|^q \, dy
\\ &= o(1) - \frac{\mathscr{D}'(s_n v_0) (s_n v_0)}{s_n^q}  \to -\infty \quad \text{as} \ n \to +\infty,
\end{align*}
thus we get a contradiction.
\end{proof}

\begin{lemma}\label{lem:7.3}
For every choice of ground states $v_n$ of $\cE_n$, the sequence $\{ v_n \}_n$ is bounded in $H$.
\end{lemma}

\begin{proof}
Suppose by contradiction that $\|v_n\| \to \infty$. Then
\begin{align*}
c_0 &= \lim_{n\to+\infty} \cE_n (v_n) = \lim_{n\to+\infty} \left( \cE_n (v_n) - \frac{1}{q} \cE_n ' (v_n)(v_n) \right) \\
&= \lim_{n\to+\infty} \left( \left( \frac{1}{2} - \frac{1}{q} \right) \|v_n\|^2 + \left( \frac{1}{q} - \frac{1}{2p} \right) \mathscr{D}'(v_n)(v_n) \right) \\
&\geq \lim_{n\to+\infty} \left( \frac{1}{2} - \frac{1}{q} \right) \|v_n\|^2 = + \infty
\end{align*}
and we obtain a contradiction.
\end{proof}

\begin{altproof}{Theorem \ref{th:main3}}
Suppose that
\begin{gather*}
\lim_{n\to+\infty} \sup_{z \in \R_+^{N+1}} \iint_{B(z,1)} |v_n|^2 \, dx \, dy = 0.
\end{gather*}
From Lion's concentration-compactness principle we obtain
\begin{gather*}
v_n \to 0 \ \mbox{in} \ L^{2(t-1)} (\R_+^{N+1}), \quad \gamma(v_n) \to 0 \ \mbox{in} \ L^t (\R^N) \quad \mbox{for all } t \in \left(2, \frac{2N}{N-1} \right).
\end{gather*}
Then, as in Lemma \ref{lem3.3} we get $\|v_n\| \to 0$. In view of Lemma \ref{lem:7.1} we have
\begin{gather*}
\cE_n(v_n) \geq \cE_n \left(r \frac{v_n}{\|v_n\|} \right) \geq a > 0
\end{gather*}
and on the other hand, in view of Lemma \ref{lem:7.3}
\begin{gather*}
\limsup_{n\to\infty} \cE_n (v_n) = - \frac{1}{2p} \limsup_{n\to\infty} \mathscr{D}(v_n) \leq 0,
\end{gather*}
a contradiction. Hence, there is a sequence $(z_n) = (x_n, y_n) \subset \mathbb{Z}_+^{N+1}$ such that
\begin{gather*}
\liminf_{n \to +\infty} \iint_{B(z_n, 1+\sqrt{N+1})} |v_n|^2 \, dx \,dy \geq \alpha
\end{gather*}
for some $\alpha > 0$. In view of Lemma \ref{lem:7.3}, there is $v \in H \setminus \{ 0 \}$ such that
\begin{align*}
v_n (\cdot + z_n) \to v \quad &\mbox{in} \ L^2_{\mathrm{loc}} (\R^{N+1}_+), \\
v_n (\cdot + z_n) \weakto v \quad &\mbox{in} \ H, \\
v_n (x+x_n, y+y_n) \to v(x,y) \quad &\mbox{for a.e.} \ (x,y) \in \R^{N+1}_+.
\end{align*}
Let $w_n := v_n(\cdot + z_n)$. Fix any $\varphi \in C_0^\infty (\R_+^{N+1})$. Observe that
\begin{align*}
\cE_0' (w_n)(\varphi) &= \cE_n ' (v_n)(\varphi(\cdot-z_n)) - \int_{\R^N} \Ga_n(y) |\gamma(v_n)|^{q-2} \gamma(v_n) \gamma(\varphi(\cdot-z_n)) \, dy \\
&= \int_{\R^N} \Ga_n(y) |\gamma(v_n)|^{q-2} \gamma(v_n) \gamma(\varphi(\cdot-z_n)) \, dy.
\end{align*}
We notice that
\begin{gather*}
\left| \int_{\R^N} \Ga_n(y) |\gamma(v_n)|^{q-2} \gamma(v_n) \gamma(\varphi(\cdot-z_n)) \, dy \right| \leq | \Gamma_n |_\infty |\gamma(w_n) |_q^{q-1} | \gamma(\varphi) |_q \to 0,
\end{gather*}
and therefore $\cE_0' (w_n)(\varphi) \to 0$. Repeating the reasoning from Step 1 in proof of Lemma \ref{lem:splitting} we can easily show that
\begin{gather*}
\cE_0' (w_n)(\varphi) \to \cE_0' (v)(\varphi)
\end{gather*}
and therefore $v$ is a nontrivial critical point of $\cE_0$. In view of Lemma \ref{lem:7.2} and Fatou's lemma we have
\begin{align}
c_0 &= \liminf_{n\to+\infty} \cE_n (v_n) = \liminf_{n\to+\infty} \left( \cE_n (v_n) - \frac{1}{2} \cE_n'(v_n)(v_n) \right) \nonumber \\
&\geq \liminf_{n\to+\infty} \left( \frac{1}{2} \frac{1}{2p} \mathscr{D}'(v_n)(v_n) - \frac{1}{2p} \mathscr{D}(v_n) \right) + \liminf_{n\to+\infty} \left( - \left( \frac{1}{2} - \frac{1}{q} \right) \int_{\R^N} \Gamma_n (y) | \gamma(v_n)|^q \, dy \right) \nonumber \\
&= \liminf_{n\to+\infty} \left( \frac{1}{2} \frac{1}{2p} \mathscr{D}'(v_n)(v_n) - \frac{1}{2p} \mathscr{D}(v_n) \right) = \liminf_{n\to+\infty}  \left(\frac{1}{2} - \frac{1}{2p}\right) \mathscr{D}(v_n) \nonumber \\
&\geq \left(\frac{1}{2} - \frac{1}{2p}\right) \mathscr{D}(v) = \frac{1}{2} \frac{1}{2p} \mathscr{D}'(v)(v) - \frac{1}{2p} \mathscr{D}(v ) \nonumber \\ &= \frac{1}{2} \frac{1}{2p} \mathscr{D}'(v)(v) - \frac{1}{2p} \mathscr{D}(v) + \frac{1}{2} \cE_0' (v)(v) = \cE_0(v) \geq c_0. \label{c0}
\end{align}
Hence $v$ is a ground state for $\cE_0$, in particular $\cE_0 (v) = c_0$. Thus it is sufficient to show that $w_n \to v$ in $H$. Observe that
\begin{align*}
\|w_n - v\|^2 &= \cE_n ' (v_n) (w_n(\cdot - z_n) - v(\cdot - z_n)) - \langle v, w_n - v \rangle \\
&\quad - \frac{1}{2p} \mathscr{D}'(w_n)(w_n-v) + \int_{\R^N} \Ga_n(y) |\gamma(w_n)|^{q-2} \gamma(w_n) \gamma(w_n - v) \, dy.
\end{align*}
We have that $\cE_n ' (v_n) (w_n(\cdot - z_n) - v(\cdot - z_n)) - \langle v, w_n - v \rangle = 0$ and $\langle v, w_n - v \rangle \to 0$. Thus
\begin{gather*}
\|w_n - v\|^2 = - \frac{1}{2p} \mathscr{D}'(w_n)(w_n-v) + \int_{\R^N} \Ga_n(y) |\gamma(w_n)|^{q-2} \gamma(w_n) \gamma(w_n - v) \, dy + o(1).
\end{gather*}
Moreover
\begin{align*}
\left| \int_{\R^N} \Ga_n(y) |\gamma(w_n)|^{q-2} \gamma(w_n) \gamma(w_n - v) \, dy \right| \leq |\Ga_n|_\infty |\gamma(w_n)|_q^{q-1} |w_n-v|_q \to 0
\end{align*}
and
\begin{gather*}
\|w_n - v\|^2 = - \frac{1}{2p} \mathscr{D}'(w_n)(w_n-v) + o(1).
\end{gather*}

From \eqref{c0} we get that
\begin{gather*}
\left(\frac{1}{2} - \frac{1}{2p}\right) \mathscr{D}(w_n) = \left(\frac{1}{2} - \frac{1}{2p}\right) \mathscr{D}(v_n) \to \left(\frac{1}{2} - \frac{1}{2p}\right) \mathscr{D}(v) = c_0.
\end{gather*}
In view of Lemma \ref{lem:brezis-D}
\begin{equation}\label{brezisLiebD-appl}
\left(\frac{1}{2} - \frac{1}{2p}\right) \mathscr{D}(w_n) - \left(\frac{1}{2} - \frac{1}{2p}\right) \mathscr{D}(w_n-v) \to \left(\frac{1}{2} - \frac{1}{2p}\right) \mathscr{D}(v) = c_0
\end{equation}
and therefore
\begin{align}\label{eq:dwv}
\mathscr{D}(w_n-v) \to 0.
\end{align}
In view of \eqref{eq:dwv} we also have
\begin{gather*}
\mathscr{D}'(w_n-v)(w_n-v) = \frac{1}{2p} \mathscr{D}(w_n-v) \to 0.
\end{gather*}
Then
\begin{gather*}
\mathscr{D}'(w_n)(w_n-v) = \mathscr{D}'(w_n)(w_n) - \mathscr{D}'(w_n)(v)
\end{gather*}
and from \eqref{brezisLiebD-appl} we have
\begin{gather*}
\mathscr{D}'(w_n)(w_n) \to \mathscr{D}'(v)(v).
\end{gather*}
Thus
\begin{gather*}
\mathscr{D}'(w_n)(w_n-v) = \mathscr{D}'(v)(v)- \mathscr{D}'(w_n)(v) + o(1).
\end{gather*}
%Observe that for any measurable set $E \subset \R^N$
%$$
%\int_E ( I_\alpha * |\gamma(w_n)|^p ) |\gamma(w_n)|^{p-2} \gamma(w_n) \gamma(v) \, dx \leq |I_\alpha|_t |\gamma(w_n)|^p_{\frac{2tp}{2t-1}} |\gamma(w_n)|^{p-1}_{\frac{2tp}{2t-1}} |\gamma(v) \chi_E|^{p-1}_{\frac{2tp}{2t-1}},
%$$
%where $t$ satisfies \eqref{1.8}. Hence the family $\{ ( I_\alpha * |\gamma(w_n)|^p ) |\gamma(w_n)|^{p-2} \gamma(w_n) \gamma(v) \}_n$ is uniformly integrable and tight. In view of the Vitali convergence theorem
%$$
%\int_{\R^N} ( I_\alpha * |\gamma(w_n)|^p ) |\gamma(w_n)|^{p-2} \gamma(w_n) \gamma(v) \, dx \to \int_{\R^N} ( I_\alpha * |\gamma(v)|^p ) |\gamma(v)|^{p} \, dx,
%$$
From Lemma \ref{weakContDprime} we have $\mathscr{D}'(w_n)(v) \to \mathscr{D}'(v)(v)$ and therefore
\begin{gather*}
\|w_n - v \|^2 \to 0.
\end{gather*}
This completes the proof.
\end{altproof}

\section*{Acknowledgements}

Bartosz Bieganowski was partially supported by the National Science Centre, Poland (Grant No. 2017/25/N/ST1/00531). Simone Secchi is member of the \emph{Gruppo Nazionale per l'Analisi Ma\-te\-ma\-ti\-ca, la Probabilit\`a e le loro Applicazioni} (GNAMPA) of the {\em Istituto Nazionale di Alta Matematica} (INdAM).
The manuscript was realized within the auspices of the INdAM -- GNAMPA Projects
\emph{Problemi non lineari alle derivate parziali} (Prot\_U-UFMBAZ-2018-000384).


\begin{thebibliography}{99}

\bibitem{a} N. Ackermann, \emph{On a periodic Schr\"{o}dinger equation with
  nonlocal superlinear part}, Math. Z. \textbf{248} (2004), 423-443.

%\bibitem{Adams} R. A. Adams, L. I. Hedberg. \emph{Function spaces and potential theory}. Grundlehren der Mathematischen Wissenschaften 314, Springer-Verlag, Berlin, 1996.

%\bibitem{ar} A. Ambrosetti, P.H. Rabinowitz, \emph{Dual variational methods
%  in critical point theory and applications}, J. Funct. Anal. \textbf{14}
%(1973), 349-381.

\bibitem{Bieganowski} B. Bieganowski, \emph{Solutions of the fractional {S}chr\"odinger equation with a
              sign-changing nonlinearity}, J. Math. Anal. Appl. \textbf{450} (2017), 461--479.

\bibitem{BieganowskiMederski} B. Bieganowski, J. Mederski, {\em Nonlinear Schr\"odinger equations with sum of periodic and vanishing potentials and sign-changing nonlinearities}, Commun. Pure Appl. Anal., Vol. 17, no 1, (2018), p. 143--161.

\bibitem{Bogachev} V.I. Bogachev, {\em Measure Theory}, Springer, Berlin, 2007.

\bibitem{CSM} X. Cabr\'e, J. Sol\`{a}-Morales,
\emph{Layers solutions in a half-space for boundary reactions}, Comm. Pure Applied Math. \textbf{58}
(2005), 1678--1732.


\bibitem{CT} X. Cabr\'e, J. Tan,
\emph{Positive solutions of nonlinear problems involving the square root of the Laplacian},
Adv. Math. \textbf{224} (2010), 2052--2093.

\bibitem{Caffarelli} L. Caffarelli, L. Silvestre,  {\em An extension problem related to the fractional Laplacian}, Comm. Partial Differential Equations \textbf{32} (2007), 1245--1260.


\bibitem{CL} Y.H. Chen, C. Liu, \emph{Ground state solutions for non-autonomous fractional {C}hoquard equations}, Nonlinearity \textbf{29} (2016), 1827--1842.


\bibitem{cho} Y. Cho, T. Ozawa, \emph{On the semirelativistic Hartree-type equation}, SIAM J. Math. Anal. \textbf{38} (2006), no.4, 1060--1074.

%	\bibitem{CS} L. Caffarelli, L. Silvestre, \emph{An extension
%            problem related to the fractional Laplacian},
%          Communications in Partial Differential Equations \textbf{32}
%          (2007), 1245--1260.


\bibitem{ccs} S. Cingolani, M. Clapp, S. Secchi, \emph{
  Multiple solutions to a magnetic nonlinear Choquard equation},
Zeitschrift f\"{u}r Angewandte Mathematik und Physik (ZAMP), \textbf{63} (2012),  233--248.



\bibitem{ccs1} S. Cingolani, M. Clapp, S. Secchi, \emph{Intertwining
  semiclassical solutions to a Schr\"{o}dinger-Newton system}, Discrete Continuous Dynmical Systems Series S \textbf{6} (2013), 891--908.

%\bibitem{ccs2} S. Cingolani, S. Secchi, \emph{Semiclassical analysis for pseudo-relativistic Hartree
%equations}, J. Differential Equations \textbf{258} (2015), 4156--4179.

\bibitem{css} S. Cingolani, S. Secchi, M. Squassina, \emph{Semiclassical
  limit for Schr\"{o}dinger equations with magnetic field and Hartree-type
  nonlinearities}, Proc. Roy. Soc. Edinburgh \textbf{140 A} (2010), 973--1009.

\bibitem{Cingolani} S. Cingolani, S. Secchi, {\em Ground states for the pseudo-relativistic Hartree
  equation with external potential}, Proc. Roy. Soc. Edinburgh Sect. A \textbf{145} (2015), 73--90.

\bibitem{CotiZelati} V. Coti Zelati, P. Rabinowitz, {\em Homoclinic type solutions for a semilinear elliptic PDE on $\R^n$}, Comm. Pure Appl. Math. \textbf{45}, (1992), no. 10, 1217--1269.


%\bibitem{CNbis} V. Coti Zelati, M. Nolasco, \emph{Ground states for pseudo-relativistic Hartree equations of critical type}, Rev. Mat. Iberoam. \textbf{29} (2013), 1421--1436.


\bibitem{Coti} V. Coti Zelati, M. Nolasco, {\em Existence of ground states for nonlinear, pseudorelativistic Schr\"odinger equations}, Red. Lincei Mat. Appl. \textbf{22} (2011), 51--72.


\bibitem{Elgart} A. Elgart, B. Schlein, \emph{Mean field dynamics of boson stars}, Comm. Pure Appl. Math., \textbf{60} (2007), 500--545.

%\bibitem{FallFelli} M. M. Fall, V. Felli, \emph{Unique continuation properties for relativistic Schr\"{o}dinger operators with a singular potential}, Discrete Contin. Dyn. Syst. \textbf{35} (2015), no. 12, 5827--5867.

%\bibitem{Felmer} P. Felmer, A. Quaas, J. Tan, \emph{Positive solutions of the nonlinear Schr\"odinger equation with the fractional Laplacian}, Proc. Roy. Soc. Edinburgh Sect. A \textbf{142} (2012), 1237--1262.

%\bibitem{FonsecaLeoni} I. Fonseca, G. Leoni. Modern methods in the Calculus of Variations: $L^p$ spaces. Springer-Verlag, 2007.


\bibitem{Frohlich} J. Fr\"{o}hlich, E. Lenzmann, \emph{Mean-field limit of
  quantum Bose gases and nonlinear Hartree equation}, in S\'{e}minaire:
\'{E}quations aux D\'{e}riv\'{e}es Partielles 2003--2004, Exp. No. XIX, 26
pp., \'{E}cole Polytech., Palaiseau, 2004.


%\bibitem{fjl} J. Fr\"{o}hlich, J. Jonsson, E. Lenzmann, \emph{Boson stars as solitary waves}, Comm. Math. Phys. \textbf{274} (2007), 1--30.


%	\bibitem{Lenz2} E. Lenzmann, \emph{Well-posedness for
%            Semi-relativistic Hartree equations with Critical type},
%          Math. Phys. Anal. Geom. \textbf{10} (2007), 43--64.

%	\bibitem{Lenz} E. Lenzmann, \emph{Uniqueness of ground states
%            for pseudo relativistic Hartree equations}, Analysis and
%          PDE \textbf{2} (2009), 1--27.

\bibitem{JeanjeanTanaka} L. Jeanjean, K. Tanaka, \emph{A positive solution for a nonlinear Schr\"odinger equation of $\R^N$}, Indiana Univ. Math. Journal, \textbf{54} (2005), 443--464.

\bibitem{LiebYau} E. H. Lieb, H.-T. Yau, \emph{The Chandrasekhar theory of stellar collapse as the limit of quantum mechanics}, Commun. Math. Phys., \textbf{112} (1987), 147--174.


%\bibitem {lieb} E.H. Lieb, \emph{Existence and uniqueness of the minimizing
%  solution of Choquard's nonlinear equation}, Stud. Appl. Math.
%\textbf{57} (1977), 93--105.

\bibitem{Lieb} E. H. Lieb, M. Loss, \emph{Analysis}, Graduate Studies in Mathematics \textbf{14}, American
Mathematical Society, 2001.


\bibitem{ls} E.H. Lieb, B. Simon, \emph{The Hartree-Fock theory for Coulomb
  systems}, Comm. Math. Phys. \textbf{53} (1977), 185--194.

%\bibitem{Lieb1983} E.H. Lieb, \emph{Sharp constants in the Hardy-Littlewood-Sobolev and related inequalities}, The Annals of Mathematics \textbf{118} (1983), 349--374.

%	\bibitem{ly} E.H. Lieb, H.-T. Yau, \emph{The Chandrasekhar
%            theory of stellar collapse as the limit of quantum
%            mechanics}, Comm. Math. Phys. \textbf{112} (1987),
%          147--174.

\bibitem{lions} P.-L. Lions, \emph{The concentration-compactness principle in the calculus of variations
. The locally compact case. Part I}, Ann. IHP, Analyse Non Lin\'{e}aire, \textbf{1} (1984), 109--145.

%\bibitem{l2} P.-L. Lions, \emph{The Choquard equation and related
%  questions}, Nonlinear Anal. T.M.A. \textbf{4} (1980), 1063--1073.


\bibitem{mz} L. Ma, L. Zhao, \emph{Classification of positive solitary
  solutions of the nonlinear Choquard equation}, Arch. Rational Mech.
Anal. \textbf{195} (2010), 455--467.



%      \bibitem{MeZo} M. Melgaard, F. Zongo, \emph{Multiple solutions
%          of the quasirelativistic Choquard equation},
%        J. Math. Phys. \textbf{53} (2012), no. 3, 033709, 12 pp.

\bibitem{mpt} I.M. Moroz, R. Penrose, P. Tod, \emph{Spherically-symmetric
  solutions of the Schr\"{o}dinger-Newton equations}, Topology of the Universe
Conference (Cleveland, OH, 1997), Classical Quantum Gravity \textbf{15}
(1998), 2733--2742.

\bibitem{Moroz} V. Moroz, J. Van Schaftingen, \emph{Groundstates of nonlinear Choquard equations: Existence, qualitative properties and decay asymptotics}, J. Funct. Anal. \textbf{265} (2013), 153--184.


% \bibitem{mt} I.M. Moroz, P. Tod, \emph{An analytical
%		 approach to the Schr\"{o}dinger-Newton
%		 equations}, Nonlinearity \textbf{12} (1999),
%		 201--216.

%	\bibitem{Mugnai} D. Mugnai, \emph{Pseudorelativistic {H}artree equation with general nonlinearity: existence, non-existence and variational identities}, Adv. Nonlinear Stud. \textbf{13} (2013), 799--823.

%\bibitem{Park} Y. J. Park, \emph{Fractional Gagliardo-Nirenberg inequality}. Journal of the ChungCheong Mathematical Society \textbf{24} (2011) n.3, 583--586.

\bibitem{pe2} R. Penrose, \emph{Quantum computation, entanglement and state
  reduction}, R. Soc. Lond. Philos. Trans. Ser. A Math. Phys. Eng. Sci.
\textbf{356} (1998), 1927--1939.

\bibitem{pe3} R. Penrose, {\em The road to reality. A complete guide to
the laws of the universe}, Alfred A. Knopf Inc., New York 2005.

%\bibitem{Ruiz} D. Ruiz, \emph{The Schr\"odinger-Poisson equation under the effect of a nonlinear local term}, J. Func. Anal. \textbf{237} (2006), 655--674.


%\bibitem{Secchi} S. Secchi: \emph{Ground state solutions for nonlinear fractional Schr\"{o}dinger equations in $\mathbb{R}^N$}, Journal of Mathematical Physics \textbf{54}, 031501 (2013); doi: 10.1063/1.4793990


\bibitem{Secchi-ground} S. Secchi, \emph{Ground state solutions for nonlinear fractional Schr\"{o}dinger equations in $\mathbb{R}^N$}, Journal of Mathematical Physics \textbf{54}, No.3, (2013), Article number 031501.

%\bibitem{Secchi1} S. Secchi, \emph{On Some Nonlinear Fractional Equations Involving the Bessel Operator}, Journal of Dynamics and Differential Equations \textbf{29} (2017), 1173-1193.

%\bibitem{Secchi2} S. Secchi, \emph{Concave-convex nonlinearities for some nonlinear fractional equations involving the Bessel operator}, Complex Variables and Elliptic Equations \textbf{62} (2017), 654--669.

%\bibitem{Sirakov} B. Sirakov, \emph{Existence and multiplicity of solutions of semmi-linear elliptic equations in $\mathbb{R}^N$}, Cal. Var. Partial Differential Equations \textbf{11} (2000), 119--142.

%\bibitem{Stein} E. M. Stein, \emph{Singular integrals
%and differentiability properties of functions}, Princeton University Press,
%Princeton, N.J., 1970.

%\bibitem{Strichartz} R. Strichartz, \emph{Analysis of the Laplacian on the complete Riemannian manifold}, J. Funct. Anal. \textbf{52} (1983), 48--79.

%	\bibitem{Tartar} L. Tartar, An introduction to Sobolev spaces. Springer, 2007.

\bibitem{Tartar} L. Tartar, \emph{An introduction to {S}obolev spaces and interpolation spaces}, Lecture Notes of the Unione Matematica Italiana \textbf{3}, Springer, Berlin; UMI, Bologna, 2007.


\bibitem{t} P. Tod, \emph{The ground state energy of the
  Schr\"{o}dinger-Newton equation}, Physics Letters A \textbf{280} (2001), 173--176.

\bibitem{ww} J. Wei, M. Winter, \emph{Strongly interacting bumps for the
  Schr\"{o}dinger-Newton equation}, J. Math. Phys. \textbf{50} (2009), Article number~012905.

%	\bibitem{W} M. Willem, Minimax theorems. Progress in
%	Nonlinear Differential Equations and their Applications,
%	24. Birkh\"{a}user Boston, Inc., Boston, MA, 1996.


\end{thebibliography}
\end{document}